\documentclass[11pt]{article}
\usepackage{amsfonts}
\usepackage{mathrsfs}

\usepackage[latin1]{inputenc}
\usepackage{amsmath,amssymb,float}
\usepackage{latexsym}
\usepackage{epstopdf}
\usepackage{geometry}   
\usepackage[active]{srcltx}
\usepackage{graphicx}
\usepackage{epstopdf}
\usepackage[
bookmarks=true,         
bookmarksnumbered=true, 
colorlinks=true, pdfstartview=FitV, linkcolor=blue, citecolor=blue,
urlcolor=blue]{hyperref}

\newtheorem {theorem}{Theorem}[section]

\newtheorem{assumption}{Assumption}
\newtheorem {corollary}{Corollary}[section]

\newtheorem{definition}{Definition}[section]
\newtheorem{example}{Example}[section]
\newtheorem{lemma}{Lemma}[section]

\newtheorem{remark}{Remark}[section]

\newenvironment{proof}[1][Proof]{\textbf{#1.} }{\
\rule{0.5em}{0.5em}}

\def\R{{\mathbb R}}
\def\E{{{\mathbb E}\,}}
\def\P{{\mathbb P}}

\def\Z{{\mathbb Z}}

\def\N{{\mathbb N}}
\def\Var{{\mathop {{\rm Var\, }}}}

\begin{document}

\begin{center}
{\LARGE  Kernel entropy estimation for linear processes}

\bigskip 

\bigskip Hailin Sang$^{a}$, Yongli Sang$^{b}$ and Fangjun Xu$^{c}$

\bigskip$^{a}$ Department of Mathematics, University of Mississippi, University, MS 38677,  USA
sang@olemiss.edu

\bigskip$^{b}$ Department of Mathematics, University of Louisiana at Lafayette, Lafayette, LA 70504, USA
yongli.sang@louisiana.edu

\bigskip$^{c}$  School of Statistics, East China Normal University, Shanghai 200241,  China 
and NYU-ECNU Institute of Mathematical Sciences at NYU Shanghai, Shanghai 200062, China
fangjunxu@gmail.com;\;fjxu@finance.ecnu.edu.cn
 \bigskip
\end{center}

\begin{abstract}
\noindent  Let $\{X_n:  n\in \N\}$ be a linear process with  bounded probability density function $f(x)$. We study the estimation of the quadratic functional $\int_{\R} f^2(x)\, dx$. With a Fourier transform on the kernel function and the projection method, it is shown that,  under certain mild conditions, the estimator
\[
\frac{2}{n(n-1)h_n} \sum_{1\le i<j\le n}K\left(\frac{X_i-X_j}{h_n}\right)
\]
has similar asymptotical properties as the i.i.d. case studied in  Gin\'{e} and Nickl (2008) if the linear process $\{X_n:  n\in \N\}$ has the defined short range dependence. We also provide an application to $L^2_2$ divergence and the extension to multivariate linear processes. The simulation study for linear processes with Gaussian and $\alpha$-stable innovations confirms our theoretical results. As an illustration, we estimate the  $L^2_2$ divergences among the density functions of average annual river flows for four rivers and obtain promising results.  
\vskip.2cm

\vskip.2cm \noindent {\bf Keywords}:
\noindent  linear process, kernel entropy estimation, quadratic functional, projection operator


\vskip.2cm 

\end{abstract}

\section{Introduction}
Let $f(x)$ be the probability density function of a sequence of identically distributed observations $\{X_n\}_{n=1}^{\infty}$. The quadratic functional $Q(f)=\int_{\R} f^2(x)\,dx$ plays an important role in the study of quadratic R\'{e}nyi entropy  $R(f)=-\ln (\int f^2(x)\,dx)$,  R\'{e}nyi (1970),   and the Shannon entropy  $S(f)=-\int f(x)\ln f(x)\,dx$,  Shannon (1948). Entropy is widely applied in the fields of information theory, statistical classification, pattern recognition and so on since it is a measure of uncertainty in a probability distribution. In the literature, different estimators for the quadratic functional and entropies with independent data  have been well studied. For example, the nearest-neighbor estimator [Leonenko, Pronzato and Savani (2008);  Penrose and Yukich (2013)],   the kernel estimator [Hall and Marron (1987); Bickel and Ritov (1988); Gin\'{e} and Nickl (2008)],  the orthogonal projection estimator [Laurent (1996, 1997)] and the $U$-statistics estimator  (a special kernel estimator) [Leonenko and Seleznjev (2010);  K\"{a}llberg,  Leonenko and Seleznjev (2012)] under the independence
assumption. It is a challenging problem to study the estimation of the quadratic functional and the corresponding entropies for dependent case.  In  K\"{a}llberg,  Leonenko and Seleznjev (2014)  the authors successfully extended the $U$-statistics method to $m$-dependence sequence. They showed the rate optimality and asymptotic normality of the $U$-statistics estimator for multivariate sequence.  Ahmad (1979) obtained the strong consistency of the quadratic functional by orthogonal series method for stationary time series with strong mixing condition. Nevertheless, to the best of our knowledge, general results for quadratic functional estimations of regular time series data under mild conditions are still unknown.  

In this paper, we study the quadratic functional $\int_{\R} f^2(x)\, dx$ for the following linear process
\begin{align}\label{lp}
X_n=\sum\limits_{i=0}^\infty a_i \varepsilon_{n-i},
\end{align}
 where the innovations $\varepsilon_i$ are independent and identically distributed
(i.i.d.)  real-valued random variables in some probability space $(\Omega, \mathcal{F}, \P)$ and $a_i$ are real coefficients such that $\sum\limits_{i=0}^\infty a_i \varepsilon_{n-i}$ converges in distribution. We would apply the kernel method which was first introduced by Rosenblatt (1956) and Parzen (1962). This method was proved to be successful in estimating the probability density functions and their derivatives, the regression functions and so on in the independent setting; see the books [Devroye and Gy\"{o}rfi (1985); Silverman (1986); Nadaraya (1989); Wand and Jones (1995);  Schimek (2000); Scott (2015)] and references therein. Kernel method was also proved to be successful in estimating the density functions and regression functions for time series data [see Tran (1992); Honda (2000); Wu and Mielniczuk (2002); Wu, Huang and Huang (2010)]. 

In the classical kernel estimation of the quadratic functional $Q(f)$ with i.i.d. observations $\{X_i\}_{i=1}^n$, Gin\'{e} and Nickl (2008) applied a convolution method  to
obtain the bias $\E T_n(h_n)-\int_{\R} f^2(x)\,dx$ of the kernel estimator 
\begin{align}\label{kernel}
T_n(h_n)=\frac{2}{n(n-1)h_n} \sum_{1\le i<j\le n}K\left(\frac{X_i-X_j}{h_n}\right).
\end{align}
They then used the Hoeffding's decomposition for $U$-statistics to study the stochastic part  $T_n(h_n)-\E T_n(h_n)$. In this way, they showed the rate optimality and efficiency of the kernel estimator for the quadratic functional. Krishnamurthy et. al. (2015) applied a similar method to study the $L^2_2$ divergence between two distributions. Nevertheless, this method does not work well in the estimation of the quadratic functional $\int_{\R} f^2(x)dx$ for linear process $X_n$ given in (\ref{lp}) due to the dependence structure. 

Instead we utilize the Fourier transform and projection methods to derive the asymptotic properties of the kernel estimator for the quadratic functional on time series data. With the help of the Fourier transform,  one can easily separate the random part $X_i-X_j$ and the kernel function $K$ in the estimator $T_n(h_n)$ given in (\ref{kernel}) and also obtain the Hoeffding's decomposition for $T_n(h_n)-\E T_n(h_n)$ when $\{X_n\}$ is a linear process; see (\ref{decomp}).  In fact, the Hoeffding's decomposition works for general stationary processes.  Then we use projection and some kind of chain rule to determine which terms contribute to the limit and which terms are negligible. Our method is different from the martingale approximation method applied in Wu and Mielniczuk  (2002) when studying the kernel density estimation for linear processes. Moreover, this method allows us to get rid of the 2nd moment assumption on the innovations of linear processes, that is, $\E|\varepsilon_1|^2<\infty$.  This allows innovations to take heavy tail distributions, which are often followed by data in finance, see Mittnik, Rachev and Paolella (1998).  We only need $\E|e^{\iota\lambda\varepsilon_1}-\phi_{\varepsilon}(\lambda)|^2\leq c_{\gamma,2}(|\lambda|^{2\gamma}\wedge 1)$ or $\E|e^{\iota\lambda\varepsilon_1}-\phi_{\varepsilon}(\lambda)|^4\leq c_{\gamma,4}(|\lambda|^{4\gamma}\wedge 1)$ for some $\gamma\in(0,1]$, where $\phi_{\varepsilon}(\lambda)$ is the characteristic function of the innovations. The range $\gamma\in(0,1]$ here is optimal when $\varepsilon_1$ has non-degenerate distribution. In this paper, the innovations have non-degenerate distribution because of the existence of the probability density function $f(x)$ of the linear process. See Lemma \ref{optimal}. On the other hand, our method gives a clear picture on the relationship between the linear process $\{X_n\}$ and the efficiency of the estimator $T_n(h_n)$. In addition, this methodology would have applications in kernel estimations. For example, we can use it to extend results in Krishnamurthy et. al. (2015) to short memory linear processes, and simplify existing proofs and assumptions in kernel density estimations. Furthermore, we could suggest a more general definition for short memory and long memory linear processes (see Definition \ref{rmk5} below) and extend our results to multivariate linear processes.

Throughout this paper, if not mentioned otherwise, the letter $c$, with or without a subscript, denotes a generic positive finite constant whose exact value is independent of $n$ and may change from line to line. We use $\iota$ to denote $\sqrt{-1}$. For a complex number $z$, we use $\overline{z}$ and $|z|$ to denote its conjugate and modulus, respectively. The notation $\|\cdot\|_2$ means $[\E |\cdot|^2]^{1/2}$. For two functions $a(x)$ and $b(x)$ of real numbers,  
$a(x)\mathbb{\sim}b(x)$ means $a(x)/b(x)\rightarrow1$ as $x\rightarrow \infty$.  For any integrable function $g(x)$, its Fourier transform is defined as $\widehat{g}(u)=\int_{\R} e^{\iota x u}\, g(x)\, dx$. Let $\phi(\lambda)$ be the characteristic function of linear process $X_n=\sum\limits^{\infty}_{i=0} a_i\varepsilon_{n-i}$, $\phi_n(\lambda)$ the empirical characteristic function from observations $\{X_i\}_{i=1}^n$ and $\phi_{\varepsilon}(\lambda)$ the characteristic function of the innovations. That is, $\phi(\lambda)=\E[e^{\iota \lambda  X_i}]$, $\phi_n(\lambda)=\frac{1}{n}\sum\limits^n_{i=1} e^{\iota \lambda X_i}$ and $\phi_{\varepsilon}(\lambda)=\E[e^{\iota \lambda  \varepsilon_1}]$.   For each $i\in\N$, define $H(X_i):=H(X_i)(\lambda):=e^{\iota \lambda X_i}-\phi(\lambda)$,  $\lambda\in\R$. 

The paper has the following structure. The main results are given in Section \ref{main}.  We then give a few examples in Section \ref{example}, an application of the results to $L^2_2$ divergence in Section \ref{application}, simulation study and real data analysis in Section \ref{simulation}, and extensions to multivariate linear processes in Section \ref{extension}. Section \ref{proof} is devoted to the proofs of Theorems \ref{thm1} and \ref{thm2}, based on the Fourier transform. 

\bigskip

\section{Main results}\label{main}
If $\{\varepsilon_{i}:\; i\in\Z\}$ is a sequence of i.i.d. random variables in $L^{p}(\R)$ for some $p>0$, $\E\varepsilon_{i}=0$ when $p\geq 1$, and
$\{a_{i}\}^{\infty}_{i=0}$ is a sequence of real coefficients such that $\sum\limits_{i=0}^{\infty}|a_{i}
|^{2\wedge p}<\infty$, then the linear process $X_n$ given in (\ref{lp}) 
exists and is well-defined. For $p=2$, the
process has short memory (short range dependence) if $\sum\limits^{\infty}_{i=0}|a_i|<\infty$ and long memory (long range dependence)  otherwise.
As for linear processes with symmetric $\alpha$-stable innovations $(0<\alpha\leq 2)$, i.e., the law of innovations having characteristic function $\E[e^{\iota\lambda\varepsilon_1}]=\exp(-c_{\alpha}|\lambda|^\alpha)$ for some positive constant $c_{\alpha}$ only depending on $\alpha$, $\{X_n\}$  has short memory if $\sum\limits_{i=0}^\infty |a_i|^{\alpha/2}<\infty$ and it has long memory if $\sum\limits_{i=0}^\infty |a_i|^{\alpha/2}=\infty$ but $\sum\limits_{i=0}^\infty |a_i|^{\alpha}<\infty$, respectively.  See, e.g., Hsing (1999). 

Let $f(x)$ be the probability density function of the linear process $X_n=\sum\limits_{i=0}^\infty a_i\varepsilon_{n-i}$, $n\in \mathbb{N}$ defined in (\ref{lp}). We study the estimation of the quadratic functional of $f(x)$, that is, $\int_{\R} f^2(x)dx$, when the linear process $X_n=\sum\limits_{i=0}^\infty a_i\varepsilon_{n-i}$ has short memory in the sense of the following definition. 
\begin{definition}\label{rmk5}
A linear process $X_n=\sum\limits_{i=0}^\infty a_i \varepsilon_{n-i}$ defined in (\ref{lp}) has short memory if 
\[
\sum\limits^{\infty}_{i=1} \sqrt{\Var(e^{\iota  \lambda a_i \varepsilon_1})}<\infty
\]
for all $\lambda$ close enough to $0$,  and long memory if 
\[
\sum\limits^{\infty}_{i=1} \sqrt{\Var(e^{\iota  \lambda a_i \varepsilon_1})}=\infty, \;\;\text{but} \;\;\sum\limits^{\infty}_{i=1} \Var(e^{\iota  \lambda a_i \varepsilon_1})<\infty
\]
for all $\lambda$ close enough to $0$.
\end{definition} 
In fact this definition generalizes the definitions of short memory or long memory linear processes
stated at the beginning of this section. Note that $\Var(e^{\iota  \lambda \varepsilon_1})$ can also be written as $\E |e^{\iota  \lambda \varepsilon_1}-\phi_\varepsilon(\lambda)|^2$ or $1-|\phi_\varepsilon(\lambda)|^2$.   By the first part of Lemma \ref{lma}, this definition contains the original one for linear processes with finite second moment innovations. Using the characteristic function of $\alpha$-stable distribution, $\E|e^{\iota \lambda \varepsilon_1}-\phi_{\varepsilon}(\lambda)|^{2}\leq c_{\frac{\alpha}{2},2} \left(|\lambda|^{\alpha}\wedge 1\right)$. Hence it also contains the one introduced in Hsing (1999) for symmetric $\alpha$-stable innovations. See the details in  Example \ref{example2}. 

Moreover, Definition \ref{rmk5} can be used to define short and long memory linear processes with general infinite variance innovations. Therefore, it can be applied to define all short and long memory linear processes with finite  variance innovations or infinite variance innovations. The advantage of Definition \ref{rmk5} is that we only need to know the behavior of $\Var(e^{\iota \lambda a_i\varepsilon_1})=1-|\phi_\varepsilon(a_i\lambda)|^2$ without any moment assumption. This definition relates the coefficients $\{a_i\}$ with the distribution of the innovations which is fully represented by the characteristic function. By contrast, the traditional definition only relates the coefficients $\{a_i\}$ with the moment information of the innovations. The definition in Hsing (1999) is only for linear processes with symmetric $\alpha$-stable innovations and it only relates the coefficients $\{a_i\}$ with the parameter $\alpha$ of the $\alpha$-stable innovations. Throughout the paper, we use Definition \ref{rmk5} to classify linear processes with short or long memory. 

To estimate the quadratic functional $\int_{\R} f^2(x)dx$, we shall apply the kernel method 
\[ 
T_n(h_n)=\frac{2}{n(n-1)h_n} \sum_{1\le i<j\le n}K\left(\frac{X_i-X_j}{h_n}\right),
\]
where the kernel $K$ is a symmetric and bounded function with $\int_{\R} K(u)\, du = 1$ and $\int_{\R} u^2|K(u)|\, du<\infty$. The bandwidth sequence $h_n$ satisfies $0<h_n\to 0$ as $n\to\infty$. The following are the main results of this paper.

\begin{theorem} \label{thm1} 
Assume that $\sum\limits^{\infty}_{i=0} |a_i|^{\gamma}<\infty$, $\int_{\R}|\lambda|^{2\gamma} |\phi_{\varepsilon}(\lambda)|^2\, d\lambda<\infty$ and 
\begin{align}\label{char4moment}
\E|e^{\iota \lambda \varepsilon_1}-\phi_{\varepsilon}(\lambda)|^{4}\leq c_{\gamma,4} \left(|\lambda|^{4\gamma}\wedge 1\right)
\end{align}
for some $\gamma\in(0,1]$. We further assume that $f$ is bounded. Then there exist positive constants $c_1$ and $c_2$ such that
\begin{align} \label{r1}
\Big|\E T_n(h_n)-\int_{\R} f^2(x)\, dx \Big|\leq c_1\left(\frac{1}{n}+h^{2\gamma}_n\right),
\end{align}
\begin{align}  \label{r2}
\E\Big(T_n(h_n)-\E T_n(h_n)-\frac{1}{n}\sum^n_{i=1}Y_i\Big)^2 \leq c_2 \Big(\frac{1}{n^2h_n}+\frac{\eta_{n,\gamma}}{n^2}+\frac{h^{2\gamma}_n}{n}\Big)
\end{align}
and, if additionally $nh_n\to \infty$ as $n\to\infty$,
\begin{align} \label{r3}
\sqrt{n}\, \Big[T_n(h_n)-\E T_n(h_n)\Big]\overset{\mathcal{L}}{\longrightarrow} N(0,4\sigma^2),
\end{align}
where $Y_i=2\big(f(X_i)-\int_{\R} f^2(x)\, dx\big)$, $\eta_{n,\gamma}=\sum\limits^{n}_{\ell=0}\sum\limits^{\infty}_{i=\ell} |a_i|^{\gamma}$
and $\sigma^2=\lim\limits_{n\to\infty} n^{-1}\Var(S_n)$ with 
\[
S_n=\sum\limits^n_{i=1}\left(f(X_i)-\int_{\R} f^2(x)\, dx\right).
\]
\end{theorem}

\begin{theorem} \label{thm2} 
Assume that $\sum\limits^{\infty}_{i=0}|a_i|^{\gamma}<\infty$, $\int_{\R}|\lambda|^{2\gamma} |\phi_{\varepsilon}(\lambda)|^2\, d\lambda<\infty$ and 
\begin{align}\label{char2moment}
\E|e^{\iota \lambda \varepsilon_1}-\phi_{\varepsilon}(\lambda)|^{2}\leq c_{\gamma,2} \left(|\lambda|^{2\gamma}\wedge 1\right)
\end{align} 
for some $\gamma\in(0,1]$. We further assume that $f$ is bounded and $\int_{\R} |\widehat{K}(\lambda)|\, d\lambda<\infty$. Then there exist positive constants $c_3$ and $c_4$ such that
\begin{align} \label{r4}
\Big|\E T_n(h_n)-\int_{\R} f^2(x)\, dx \Big|\leq c_3\left(\frac{1}{n}+h^{2\gamma}_n\right),
\end{align}
\begin{align}  \label{r5}
\E\Big|T_n(h_n)-\E T_n(h_n)-\frac{1}{n}\sum^n_{i=1}Y_i\Big| \leq c_4 \Big(\frac{1}{nh_n}+\frac{h^{\gamma}_n}{\sqrt{n}}\Big)
\end{align}
and, if additionally $\sqrt{n}h_n\to\infty$ as $n\to\infty$,
\begin{align} \label{r6}
\sqrt{n}\, \Big[T_n(h_n)-\E T_n(h_n)\Big]\overset{\mathcal{L}}{\longrightarrow} N(0,4\sigma^2)
\end{align}
for some $\sigma^2\in(0,\infty)$.

\end{theorem}

\begin{remark}\label{rmk0}
By Definition \ref{rmk5}, Theorems \ref{thm1} and \ref{thm2} are results for short memory linear processes. On the other hand, we work on the well-defined linear process (\ref{lp}) in Theorem \ref{thm1}, Theorem \ref{thm2} and the rest of the paper. If we assume $\varepsilon_1\in L^\gamma(\mathbb{R})$, together with condition $\sum\limits^{\infty}_{i=0}|a_i|^{\gamma}<\infty$ in the theorems, then the linear process (\ref{lp}) is well-defined. However, the moment condition $\varepsilon_1\in L^\gamma(\R)$ is not needed in the proofs of our results. 
\end{remark}

\begin{remark} \label{rmk1} When $\{X_n\}$ is an i.i.d. sequence, the kernel estimator (\ref{kernel}) was studied in Gin\'{e} and Nickl (2008) with assumptions that  the density function $f(x)$ is bounded and the smoothness condition $\int_{\R}|\lambda|^{2\gamma} |\phi(\lambda)|^2\, d\lambda<\infty$ of $f(x)$ for  $\gamma$ restricted to $(0,\frac{1}{2}]$. If $\int_{\R}u^2|K(u)|\, du<\infty$ is replaced by a weaker condition $\int_{\R} |u|^{\beta}|K(u)|\, du<\infty$ for some $0<\beta\le 2$, then the upper bounds in (\ref{r1}) and (\ref{r4}) above should be replaced by a constant multiple of $\frac{1}{n}+h^{\beta\wedge 2\gamma}_n$. 
\end{remark}

\begin{remark} \label{rmk2} The assumption (\ref{char4moment}) in Theorem \ref{thm1} is required to obtain (\ref{r2}) and (\ref{r3}). To just get the inequality (\ref{r1}), we only need to assume 
(\ref{char2moment}).  Using Lemma \ref{lma}, we could easily see that $\E|\varepsilon_1|^{4\gamma}<\infty$ and $\E|\varepsilon_1|^{2\gamma}<\infty$ are sufficient for these two assumptions, respectively. However, these two moment conditions are not necessary. For example, when the innovation $\varepsilon_1$ takes symmetric $\alpha$-stable distributions with $\alpha\in(0,2)$,  $\E|e^{\iota \lambda \varepsilon_1}-\phi_{\varepsilon}(\lambda)|^{4}\leq c_{\frac{\alpha}{4},4} \left(|\lambda|^{\alpha}\wedge 1\right)$, $\E|e^{\iota \lambda \varepsilon_1}-\phi_{\varepsilon}(\lambda)|^{2}\leq c_{\frac{\alpha}{2},2} \left(|\lambda|^{\alpha}\wedge 1\right)$, but $\E|\varepsilon_1|^{\alpha}=\infty$. On the other hand, assumption (\ref{char4moment}) is stronger than assumption (\ref{char2moment}). This explains why the results in Theorem \ref{thm1} are better than the results in Theorem \ref{thm2} if both assumptions (\ref{char4moment}) and (\ref{char2moment}) hold for a same $\gamma$ in $(0,1]$. 
\end{remark}

\begin{remark}\label{rmk3} If the underlying linear process $\{X_n\}$ is i.i.d. or $m$-dependent, then the assumption $\E|e^{\iota \lambda \varepsilon_1}-\phi_{\varepsilon}(\lambda)|^{4}\leq c_{\gamma,4} \left(|\lambda|^{4\gamma}\wedge 1\right)$ for some $\gamma\in(0,1]$ in Theorem \ref{thm1} can be removed by making some minor modifications of its proof, and $\eta_{n,\gamma}$ can be replaced by $0$. Furthermore, for the i.i.d. case, the right hand side of (\ref{r1}) can be replaced by $c_1h^{2\gamma}_n$.
\end{remark}

\begin{remark}\label{rmk4} In Theorems \ref{thm1} and \ref{thm2}, if we further have $\int_{\R}|\lambda|^{4\gamma} |\phi_{\varepsilon}(\lambda)|^2\, d\lambda<\infty$, then, using the fact $|\widehat{K}(\lambda h_n)-\widehat{K}(0)|\leq c_{\beta}|\lambda h_n|^{2\beta}$ for all $\beta\in[0,1]$,  the inequality (\ref{r2}) can be improved to  
\begin{align*} 
\E\Big|T_n(h_n)-\E T_n(h_n)-\frac{1}{n}\sum^n_{i=1}Y_i\Big|^2 \leq c_2 \Big(\frac{1}{n^2h_n}+\frac{\eta_{n,\gamma}}{n^2}+\frac{h^{4\gamma}_n}{n}\Big)
\end{align*}
while (\ref{r5}) to
\begin{align*} 
\E\Big|T_n(h_n)-\E T_n(h_n)-\frac{1}{n}\sum^n_{i=1}Y_i\Big| \leq c_4 \Big(\frac{1}{nh_n}+\frac{h^{2\gamma}_n}{\sqrt{n}}\Big).
\end{align*}
\end{remark}

As consequences of Theorem \ref{thm1} and Theorem \ref{thm2}, we can easily obtain the following results.
\begin{corollary} \label{corollary1}
Under the assumptions in Theorem \ref{thm1}. Let $h_n$ have order $n^{-\frac{2}{4\gamma+1}}$. If $0 < \gamma \le 1/4$,  then 
\begin{align} \label{opt1}
T_n(h_n)-\int_{\R} f^2(x)\,dx=O_P(n^{-\frac{4\gamma}{4\gamma+1}});
\end{align}
if $1/4<\gamma\leq 1$, then
\begin{align} \label{clt1}
\sqrt{n}\, \Big[T_n(h_n)-\int_{\R} f^2(x)\,dx\Big]\overset{\mathcal{L}}{\longrightarrow} N(0,4\sigma^2),
\end{align}
and, if we additional require $K(x)\geq 0$ for all $x$, the estimator  $-\ln (\frac{1}{n}+T_n(h_n))$ of the quadratic R\'{e}nyi entropy  $R(f)=-\ln (\int_{\R} f^2(x)\,dx)$ satisfies 
\begin{align} \label{cltr1}
\sqrt{n}\, \Big[-\ln (\frac{1}{n}+T_n(h_n))-R(f)\Big]\overset{\mathcal{L}}{\longrightarrow} N\left(0,\frac{4\sigma^2}{(\int_{\R} f^2(x)\,dx)^2}\right),
\end{align}
where $\sigma^2=\lim\limits_{n\to\infty} n^{-1}\Var(S_n)$ with $S_n=\sum\limits^n_{i=1}\left(f(X_i)-\int_{\R} f^2(x)\, dx\right)$.
\end{corollary} 

\begin{corollary} \label{corollary2}
Under the assumptions in Theorem \ref{thm2}. Let $h_n$ have order $n^{-\frac{1}{2\gamma+1}}$. If $0 < \gamma \le 1/2$,  then 
\begin{align} \label{opt2}
T_n(h_n)-\int_{\R} f^2(x)\,dx=O_P(n^{-\frac{2\gamma}{2\gamma+1}});
\end{align}
if $1/2<\gamma\leq 1$, then
\begin{align} \label{clt2}
\sqrt{n}\, \Big[T_n(h_n)-\int_{\R} f^2(x)\,dx\Big]\overset{\mathcal{L}}{\longrightarrow} N(0,4\sigma^2),
\end{align}
and, if we additional require $K(x)\geq 0$ for all $x$,  the estimator  $-\ln (\frac{1}{n}+T_n(h_n))$ of the quadratic R\'{e}nyi entropy  $R(f)=-\ln (\int_{\R} f^2(x)\,dx)$ satisfies 
\begin{align} \label{cltr2}
\sqrt{n}\, \Big[-\ln (\frac{1}{n}+T_n(h_n))-R(f)\Big]\overset{\mathcal{L}}{\longrightarrow} N\left(0,\frac{4\sigma^2}{(\int_{\R} f^2(x)\,dx)^2}\right)
\end{align}
for some $\sigma^2\in(0,\infty)$.
\end{corollary}

\bigskip

\section{Examples}\label{example}

In this section, we shall give a few examples of linear processes to illustrate our results.

\begin{example}[Linear process with Gaussian white noise innovations] Here the distribution of $\varepsilon_1$ is Gaussian with mean $0$ and variance $\sigma^2$. Note that $\E[e^{\iota\lambda \varepsilon_1}]=\exp(-\frac{\lambda^2}{2\sigma^2})$. So $\int_{\R}\lambda^2 |\phi_{\varepsilon}(\lambda)|^2\, d\lambda<\infty$, $\E|e^{\iota \lambda \varepsilon_1}-\phi_{\varepsilon}(\lambda)|^{4}\leq c_{1,4} \left(\lambda^4\wedge 1\right)$ and $\E|e^{\iota \lambda \varepsilon_1}-\phi_{\varepsilon}(\lambda)|^{2}\leq c_{1,2} \left(\lambda^{2}\wedge 1\right)$. Therefore, Theorems \ref{thm1}, \ref{thm2} and Corollaries \ref{corollary1} and \ref{corollary2} hold for all short memory linear processes with Gaussian white noise innovations.\end{example}

\begin{example}\label{example2}{\bf (Linear process with symmetric $\alpha$-stable innovations)} Here the distribution of $\varepsilon_1$ is the symmetric $\alpha$-stable (S$\alpha$S) with $0<\alpha<2$. Note that $\E[e^{\iota\lambda \varepsilon_1}]=\exp(-c_{\alpha}|\lambda|^\alpha)$ for some positive constant $c_{\alpha}$. So $\int_{\R}\lambda^2 |\phi_{\varepsilon}(\lambda)|^2\, d\lambda<\infty$, $\E|e^{\iota \lambda \varepsilon_1}-\phi_{\varepsilon}(\lambda)|^{4}\leq c_{\frac{\alpha}{4},4} \left(|\lambda|^{\alpha}\wedge 1\right)$ and $\E|e^{\iota \lambda \varepsilon_1}-\phi_{\varepsilon}(\lambda)|^{2}\leq c_{\frac{\alpha}{2},2} \left(|\lambda|^{\alpha}\wedge 1\right)$. Therefore, Theorems \ref{thm1} and \ref{thm2} hold for all short memory linear processes having symmetric $\alpha$-stable innovations ($0<\alpha<2$) with $\gamma=\frac{\alpha}{4}$ and $\gamma=\frac{\alpha}{2}$, respectively.  
The central limit theorem (\ref{clt1}) in Corollary \ref{corollary1} and the central limit theorem (\ref{clt2}) in Corollary \ref{corollary2} hold if $1<\alpha<2$. We should apply Theorem \ref{thm2} and Corollary \ref{corollary2} instead of Theorem \ref{thm1} and Corollary \ref{corollary1}  in order to study all short-memory processes with $\alpha$-stable innovations and coefficients satisfying $\sum_i |a_i|^{\alpha/2}<\infty$. See Step 2 of Section \ref{simulation} for the detailed discussion.  
\end{example}

\begin{example}[Causal $ARMA(p, q)$ process]  An $ARMA(p, q)$ model has the form 
\[
\phi(B)X_t=\theta(B)\varepsilon_t,
\]
where $\varepsilon_t$ are i.i.d random variables with $\E \varepsilon_t=0$ and $\E \varepsilon^2_t=\sigma^2_{\varepsilon}$, $\phi$ and $\theta$ are two polynomials with no common zeros,
\begin{align*}
\phi(z)&=1-\phi_1 z-\phi_2 z^2-\cdots-\phi_p z^p, \;\; \phi_p\ne 0,\\
\theta(z)&=1+\theta_1 z+\theta_2 z^2+\cdots-\theta_q z^q, \;\;\theta_q\ne 0,
\end{align*}
and $B$ is the back shift operator with $B^k X_t=X_{t-k}$.
An $ARMA(p, q)$ process is causal if it can be written as a one-sided linear process $X_t=\sum^\infty\limits_{i=0} a_i\varepsilon_{t-i}$ with $\sum^\infty\limits_{i=0} |a_i|<\infty$. It is causal if and only if $\phi(z)\ne 0$ for $|z|\le 1$. In this case, the coefficients of the linear process can be determined by solving 
\[
a(z)=\sum^\infty\limits_{i=0} a_i z^i=\frac{\phi(z)}{\theta(z)},
\] 
where $|z|\le 1$.  It is easy to see that the coefficients $a_i$ decay exponentially fast to zero. So $\sum\limits^{\infty}_{\ell=0}\sum\limits^{\infty}_{i=\ell} |a_i|^{\gamma}<\infty$ for any $\gamma>0$. This means that we can choose $\gamma$ as small as possible, which will enlarge the choices of the innovations $\varepsilon$. Moreover, the upper bound in (\ref{r2}) of Theorem \ref{thm1} can be improved to a constant multiple of $\frac{1}{n^2h_n}+\frac{h^{2\gamma}_n}{n}$.

\end{example}

\bigskip

\section{An application to $L^2_2$ divergence}\label{application}

Let $f(x)$ and $g(x)$ be the probability densities for independent linear processes $X_n=\sum\limits_{i=0}^\infty a_i \varepsilon_{n-i}$ and $Y_n=\sum\limits_{j=0}^\infty b_j\xi_{n-j}$ defined as in $(\ref{lp})$, respectively. Then the $L_2^2$ divergence between $f$ and $g$ is defined as 
\[
D(f,g):=\int_{\R} (f(x)-g(x))^2\, dx=\int_{\R} f^2(x)\, dx+\int_{\R} g^2(x)\, dx-2\int_{\R} f(x)g(x)\, dx.
\]
The estimator for $D(f,g)$ is 
\[
\widehat{D}(f,g)=T_n(f, h_n)+T_n(g, h_n)-2T_n(f, g, h_n),
\] 
where 
\[
T_n(f, h_n)=\frac{2}{n(n-1)h_n} \sum_{1\le i<j\le n}K\left(\frac{X_i-X_j}{h_n}\right),
\]
\[
T_n(g, h_n)=\frac{2}{n(n-1)h_n} \sum_{1\le i<j\le n}K\left(\frac{Y_i-Y_j}{h_n}\right)
\]
and 
\[
T_n(f, g, h_n)=\frac{1}{n^2h_n} \sum_{1\le i, j\le n}K\left(\frac{X_i-Y_j}{h_n}\right).
\]
Krishnamurthy et. al. (2015) studied the estimation of $L_2^2$ divergence for multivariate independent i.i.d. samples $\{X_i\}_{i=1}^{2n}\sim f(x)$ and $\{Y_i\}_{i=1}^{2n}\sim g(x)$. They used data splitting in the estimation which is necessary to obtain the asymptotic normality [see Gretton et. al. (2012);  Krishnamurthy et. al. (2015)].  However, in our case, the data splitting is not needed. 

Using similar arguments as in the proof of Theorems \ref{thm1}  and \ref{thm2}, one can obtain the same results for $T_n(f, g, h_n)$. Together with Theorems \ref{thm1} and \ref{thm2}, we then have the following results. 
\begin{theorem} \label{thm3} 
Assume that the  independent linear processes $\{X_n\}$ and $\{Y_n\}$ satisfy the same conditions as in Theorem \ref{thm1}. Then there exist positive constants $c_1$ and $c_2$ such that
\begin{align} \label{r1'}
\Big|\E \widehat{D}(f,g)-D(f,g)\Big|\leq c_1 \left(\frac{1}{n}+ h^{2\gamma}_n\right),
\end{align}
\begin{align}  \label{r2'}
\E\Big|\widehat{D}(f,g)-\E \widehat{D}(f,g)-\frac{1}{n}\sum^n_{i=1}W_i\Big|^2 \leq c_2 \Big(\frac{1}{n^2h_n}+\frac{\eta_{n,\gamma}}{n^2}+\frac{h^{2\gamma}_n}{n}\Big)
\end{align}
and, if additionally $n h_n\to \infty$ as $n\to\infty$,
\begin{align} \label{r3'}
\sqrt{n}\, \Big[\widehat{D}(f,g)-\E \widehat{D}(f,g)\Big]\overset{\mathcal{L}}{\longrightarrow} N(0,4\sigma^2),
\end{align}
where 
\[
W_i=2\Big(f(X_i)-\int_{\R} f^2(x)\, dx+g(Y_i)-\int_{\R} g^2(x)\, dx+g(X_{i})+f(Y_{i})-2\int_{\R} f(x)g(x)\, dx\Big),
\] 
$\eta_{n,\gamma}=\sum\limits_{\ell=0}^n\sum\limits_{i=\ell}^\infty |a_i|^\gamma\vee \sum\limits_{\ell=0}^n\sum\limits_{i=\ell}^\infty |b_i|^\gamma$, and $\sigma^2=\lim\limits_{n\to\infty} n^{-1}\Var(S_n)$ with $S_n=\frac{1}{2}\sum\limits^n_{i=1}W_i$.  
\end{theorem}

Let $h_n$ have order $n^{-\frac{2}{4\gamma+1}}$. Then similar results as in Corollary \ref{corollary1} can be obtained for $\widehat{D}(f,g)$.

\begin{theorem} \label{thm4} 
With the same notations as in Theorem \ref{thm3}, assume that the independent linear processes $\{X_n\}$ and $\{Y_n\}$ satisfy the same conditions as in Theorem \ref{thm2}. Then there exists a positive constant $c$ such that 
\begin{align}  
\E\Big|\widehat{D}(f,g)-\E \widehat{D}(f,g)-\frac{1}{n}\sum^n_{i=1}W_i\Big| \leq c \Big(\frac{1}{nh_n}+\frac{h^{\gamma}_n}{\sqrt{n}}\Big)
\end{align}
and (\ref{r1'}), (\ref{r3'}) hold for some $\sigma^2\in(0,\infty)$. 
\end{theorem}

Let $h_n$ have order $n^{-\frac{1}{2\gamma+1}}$. Then similar results as in Corollary \ref{corollary2} can be obtained for $\widehat{D}(f,g)$.
\bigskip

\section{Simulation study and real data analysis}\label{simulation}

In this section we first conduct a simulation study to verify the central limit theorems in Section \ref{main} for the kernel entropy estimator of linear processes with different coefficient sequences and innovations with stable or Gaussian laws.  For this purpose, we shall draw normal quantile plots and histograms, perform hypothesis tests and provide confidence intervals on the means. In the second part of this section, we apply the kernel method to study the divergences between each pair of the annual average flow distributions of four rivers on the earth. 

The simulation study has two steps. 

{\bf Step 1}: Calculate or approximate the true value of the quadratic functional $\int_\mathbb{R} f^2(x) dx$.

By the inverse Fourier transform, the probability density function $f(x)$ of the linear process in (\ref{lp}) and its characteristic function $\phi(t)=\mathbb{E}[e^{\iota tX_1}]$ have the following relationship:
\begin{align}\label{inverseF}
f(x)=\frac{1}{2\pi}\int_{\mathbb{R}}\phi(t)e^{-\iota tx}dt.
\end{align}
Also, we have 
\begin{align*}
\phi(t)=\mathbb{E}[e^{\iota tX_n}]=\prod_{i=0}^\infty\mathbb{E}[e^{\iota t a_i \varepsilon_{n-i}}]. 
\end{align*}
Now let  $a_i=\frac{\Gamma(i+d)}{\Gamma(d)\Gamma(i+1)}$, $i\ge 0$, for any $d<\frac{1}{2}, d\ne 0, -1, -2, \cdots$. Applying Stirling's formula, $\Gamma(x)\sim \sqrt{2\pi}e^{-x+1}(x-1)^{x-1/2}$ as $x\rightarrow \infty$, $a_i\sim i^{d-1}/\Gamma(d)$ as $i\rightarrow \infty$. In particular, under the condition $-1<d<\frac{1}{2}, d\ne 0$,  the FARIMA$(0,d,0)$ processes $X_n=(1-B)^{-d}\varepsilon_n$,  where $B$ is the backshift operator, are causal and they have the form of linear process with $a_i=\frac{\Gamma(i+d)}{\Gamma(d)\Gamma(i+1)}$.  See Bondon and Palma (2007) for the extension of causality of FARIMA$(p,d,q)$ processes to the range of $-1<d<1/2$. We shall perform simulation study for linear processes with Gaussian or Cauchy innovations and some selected values of $d$. 

{\it Case 1}. Suppose the innovations $\{\varepsilon_i\}$ are i.i.d. $N(0,1)$ random variables, then $\phi_\varepsilon(t)=e^{-\frac{t^2}{2}}$ and hence
$\phi(t)=e^{-\frac{t^2}{2}\sum_{i=0}^\infty a_i^2}$.   By the Gauss's theorem [Gauss (1866)] for hypergeometric series,  
$\sum_{i=0}^{\infty} a^2_i=\frac{\Gamma(1-2d)}{\Gamma^2(1-d)}$ for any $d<\frac{1}{2}, d\ne 0, -1, -2, \cdots$. See also Bailey (1935) or its direct calculation in Sang and Sang (2017).    
By the formula (\ref{inverseF}) and the application of MATLAB,
\begin{align*}
f(x)=\frac{1}{\pi}\int_{0}^{\infty}e^{-\frac{t^2}{2}\frac{\Gamma(1-2d)}{\Gamma^2(1-d)}} \cos (tx)dt=\frac{\sqrt{2}}{4}e^{-\frac{\pi x^2}{8}}
\end{align*}
if we take $d=-0.5$.
As a result,
\begin{align*}
\int_\mathbb{R} f^2(x)dx=0.2500\;\;\;\text{if $d=-0.5$}.
\end{align*}

In this way, we also calculate or approximate the value of $\int_\mathbb{R} f^2(x)dx$ for other $d$ values as listed in Table \ref{tab1}. 
\begin{table}[H]
\center
\caption{ The approximate value of $\int_\mathbb{R} f^2(x)dx$ for linear processes with Gaussian innovations} 
\label{tab1}
\bigskip
\begin{tabular}{| l | l | l | l | l | l | l | l | l |}
\hline                                                     
\;\;$d$                                          & $-0.1$  & $-0.3$ &  $-0.5$ &  $-0.7$ &  $-0.9$ &  $-1.5$&  $-2.5$ &  $-3.5$\\
\hline

\;\;$\int_\mathbb{R} f^2(x)dx$ & 0.2801 & 0.2678 & 0.2500 & 0.2300 & 0.2095 &0.1531 &  0.0856 & 0.0462 \\
\hline
\end{tabular}  
\end{table}


{\it Case 2}. If the innovations $\{\varepsilon_i\}$ have i.i.d. symmetric $\alpha$-stable distribution with $0<\alpha<2$, then $\phi_\varepsilon(t)=e^{-|t|^\alpha}$ and 
$\phi(t)=e^{-|t|^\alpha\sum_{i=0}^\infty |a_i|^\alpha}$.
In particular, if $\alpha=1$, we have the standard Cauchy distribution with $\phi_\varepsilon(t)=e^{-|t|}$ and 
$\phi(t)=e^{-|t|\sum_{i=0}^\infty |a_i|}$. In this case, with the help of the software {\it Mathematica}, we approximate $\sum_{i=0}^\infty |a_i|$ by  $\sum_{i=0}^{100,000} |a_i|$ and then calculate $f(x)$ and the value of $\int_\mathbb{R} f^2(x)dx$. The result is in Table \ref{tab2}. 
 
 \begin{table}[H]
\center
\caption{ The approximate value of $\int_\mathbb{R} f^2(x)dx$ for linear processes with Cauchy innovations} 
\label{tab2}
\bigskip
\begin{tabular}{| l | l | l | l | l | l |}
\hline                                                     
\;\;$d$                                          & $-0.1$  & $-0.3$ &  $-0.5$ &  $-0.7$ &  $-0.9$ \\
\hline

\;\;$\int_\mathbb{R} f^2(x)dx$ & 0.0934 & 0.0806 & 0.0796 & 0.0796 & 0.0796 \\
\hline
\end{tabular}  
\end{table}

{\bf Step 2}: For each case, linear processes with Gaussian or Cauchy innovations, and each memory parameter $d$, we produce $m=1,000$ linear processes with $n=4,096$ observations $\{X_i\}_{i=1}^n$ in each of them by applying the MATLAB code in Fa\"{y} et. al. (2009) with some modification.  For each generated linear process, we apply the estimator $T_n(h_n)$ as in (\ref{kernel}) to estimate the true quadratic functional $\int_\mathbb{R} f^2(x) dx$. The estimation is realized in MATLAB. Here we choose the normal kernel, that is, $K(x)=\frac{1}{\sqrt{2\pi}}e^{-\frac{x^2}{2}}$, in the estimator $T_n(h_n)$. The other common kernel functions are also applicable since the main theorems and corollaries only assume  some basic conditions on the kernel functions. 

 For the chosen  $a_i=\frac{\Gamma(i+d)}{\Gamma(d)\Gamma(i+1)}$, $i\ge 0$, and for all results in Section \ref{main},  we require $\gamma>\frac{1}{1-d}$ since $\sum_{i=0}^\infty |a_i|^\gamma<\infty$ and $a_i\sim i^{d-1}/\Gamma(d)$.
 
 {\it Case 1}. Suppose the innovations $\varepsilon_i$, $i\in\mathbb{Z}$, are i.i.d. $N(0,1)$ random variables. Then $\E[e^{\iota\lambda \varepsilon_1}]=\exp(-\frac{\lambda^2}{2\sigma^2})$. So $\int_{\R}\lambda^2 |\phi_{\varepsilon}(\lambda)|^2\, d\lambda<\infty$, $\E|e^{\iota \lambda \varepsilon_1}-\phi_{\varepsilon}(\lambda)|^{4}\leq c_{1,4} \left(\lambda^4\wedge 1\right)$ and $\E|e^{\iota \lambda \varepsilon_1}-\phi_{\varepsilon}(\lambda)|^{2}\leq c_{1,2} \left(\lambda^{2}\wedge 1\right)$. 
 In this case, we should apply Theorem \ref{thm1} and Corollary \ref{corollary1} instead of Theorem \ref{thm2} and Corollary \ref{corollary2}. 
 The convergence rate (\ref{opt1}) in Corollary \ref{corollary1} is faster than the convergence rate (\ref{opt2}) in Corollary \ref{corollary2}. The central limit theorem (\ref{clt1}) in Corollary \ref{corollary1} holds for $1/4<\gamma \le 1$ while the central limit theorem (\ref{clt2}) in Corollary \ref{corollary2} holds only  for $1/2<\gamma \le 1$. See also Remark \ref{rmk2}. To have the central limit theorem (\ref{clt1}), we require $\frac{1}{1-d}\vee \frac{1}{4}<\gamma\le 1$ and the bandwidth  $h_n$ to have the order of $n^{-\frac{2}{4\gamma+1}}$. In this Gaussian case, we recommend to take $\gamma=1$ and $h_n$ to have the order of $n^{-\frac{2}{5}}$ to make (\ref{clt1}) hold for all $d<0$. To apply kernel method in estimation, one should select an optimal bandwidth based on some criteria, for example, to minimize the mean squared error. For the i.i.d. case, Gin\'{e} and Nickl (2008) construct a kernel-based rate adaptive estimator of $\int_\mathbb{R}f^2(x)dx$ by selecting the bandwidth $h_n$ adaptively in the first step. It is interesting to investigate the bandwidth adaptive kernel estimator of  $\int_\mathbb{R}f^2(x)dx$ for linear processes from both theoretical and application viewpoints. However, it seems that the study in this direction is very difficult. We leave it as an open question for future study. For the current simulation study, we take $h_n=n^{-2/5}$ in the Gaussian case. 
 
 {\it Case 2}.  In the $\alpha$-stable case, $\int_{\R}\lambda^2 |\phi_{\varepsilon}(\lambda)|^2\, d\lambda<\infty$, $\E|e^{\iota \lambda \varepsilon_1}-\phi_{\varepsilon}(\lambda)|^{4}\leq c_{\frac{\alpha}{4},4} \left(|\lambda|^{\alpha}\wedge 1\right)$ and $\E|e^{\iota \lambda \varepsilon_1}-\phi_{\varepsilon}(\lambda)|^{2}\leq c_{\frac{\alpha}{2},2} \left(|\lambda|^{\alpha}\wedge 1\right)$. Therefore, Theorem \ref{thm1} (Corollary \ref{corollary1}) and Theorem \ref{thm2} (Corollary \ref{corollary2}) hold for all short memory linear processes having symmetric $\alpha$-stable innovations ($0<\alpha<2$) with $\gamma=\frac{\alpha}{4}$ and $\gamma=\frac{\alpha}{2}$, respectively. Hence we require $\sum_{i=0}^\infty |a_i|^{\alpha/4}<\infty$ in Theorem \ref{thm1} and Corollary \ref{corollary1} and we only require $\sum_{i=0}^\infty |a_i|^{\alpha/2}<\infty$ in Theorem \ref{thm2} and Corollary \ref{corollary2}.  For this reason, we should apply 
Theorem \ref{thm2} and Corollary \ref{corollary2} instead of Theorem \ref{thm1} and Corollary \ref{corollary1}. Since $\gamma=\frac{\alpha}{2}$, Theorem \ref{thm2} holds  for all $\alpha\in (0,2)$ while in Corollary \ref{corollary2}, we have (\ref{opt2}) for $\alpha\in (0,1]$ and the central limit theorem (\ref{clt2}) only for $\alpha\in (1,2)$.
%
Hence if $1<\alpha<2$, we should take $h_n$ to have the order of $n^{-\frac{1}{\alpha+1}}$ to confirm the central limit theorem (\ref{clt2}).  If  $0<\alpha\le 1$, we should not expect to have (\ref{clt2}). But in this case we can take $h_n$ with $\sqrt{n}h_n\rightarrow \infty$ to confirm the central limit theorem (\ref{r6}) in Theorem \ref{thm2}. In the simulation study, we take $h_n=n^{-2/5}$. In Theorem \ref{thm2} and Corollary \ref{corollary2}, the coefficients $a_i$ should have parameter $d$ with $d<1-\frac{1}{\gamma}=1-\frac{2}{\alpha}$.  

Figure \ref{fig:1} and Figure \ref{fig:2} show the histograms with kernel density fits and normal Q-Q plots of $r_n=\sqrt{n}\, \Big[T_n(h_n)-\int_{\R} f^2(x)\,dx\Big]$ from $m=1,000$ simulated linear processes with Gaussian innovations and $d=-0.1,-0.9, -3.5$.  It is clear that the realizations of $r_n$ are distributed as some normal distribution in each case. We also have performed simulation study for $d=-0.3, -0.5,-0.7, -1.5$ and $-2.5$ and obtained similar plots and results.

We then perform hypothesis test $H_0:\mu=0$;  $H_1:\mu\ne 0$ for the mean of the asymptotic distribution of $r_n$ and provide the 95\% confidence interval for $\mu$ for each value of $d$. The result is listed in Table \ref{tab3}. It is clear that $r_n$ is  asymptotically centered  in each case. Over all, the simulation study here confirms the analysis for the Gaussian case based on the main results in Section \ref{main}. 
\begin{figure}[H]
\centering
\begin{tabular}{cc}
\includegraphics[width=0.33\linewidth]{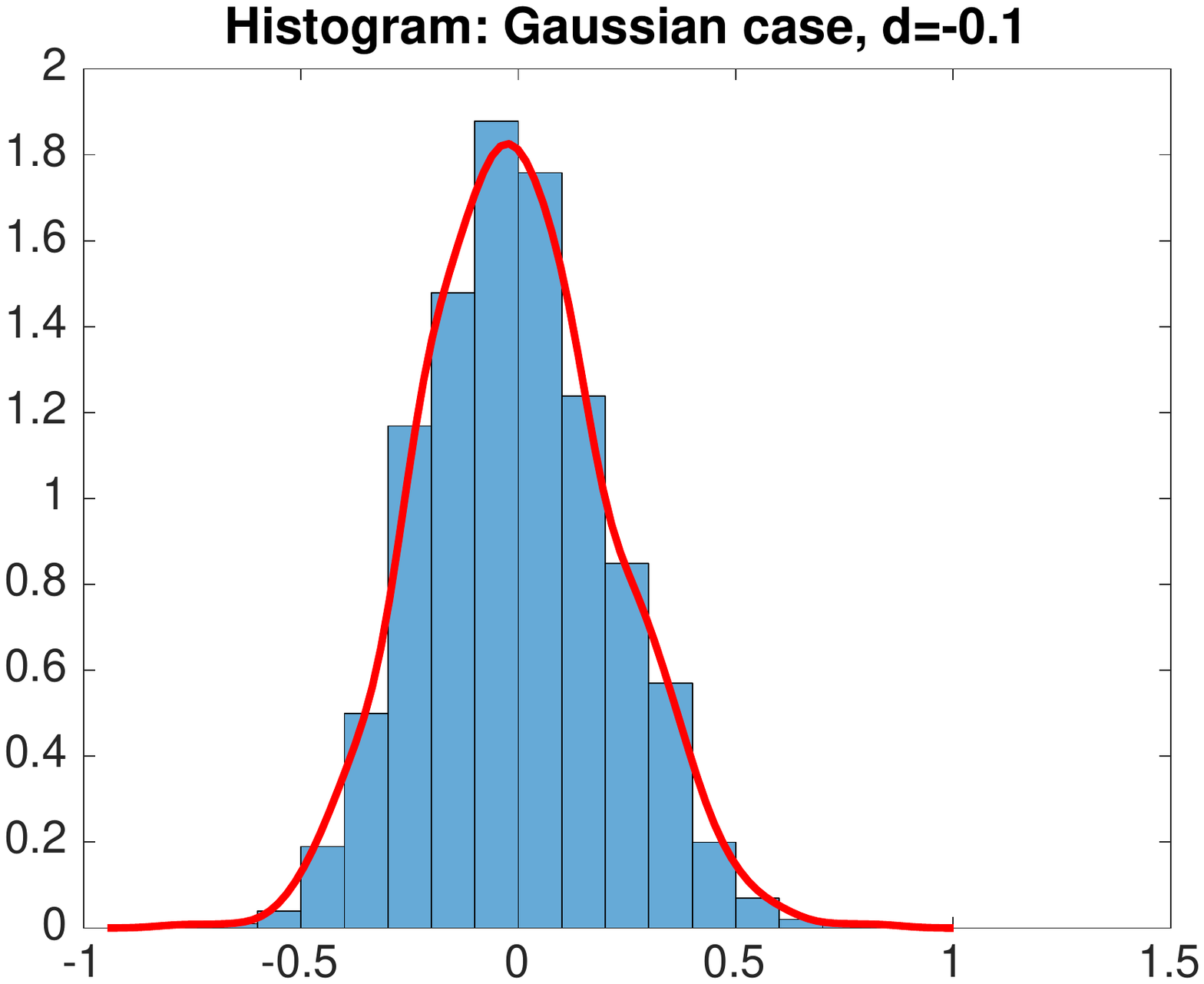}
\includegraphics[width=0.33\linewidth]{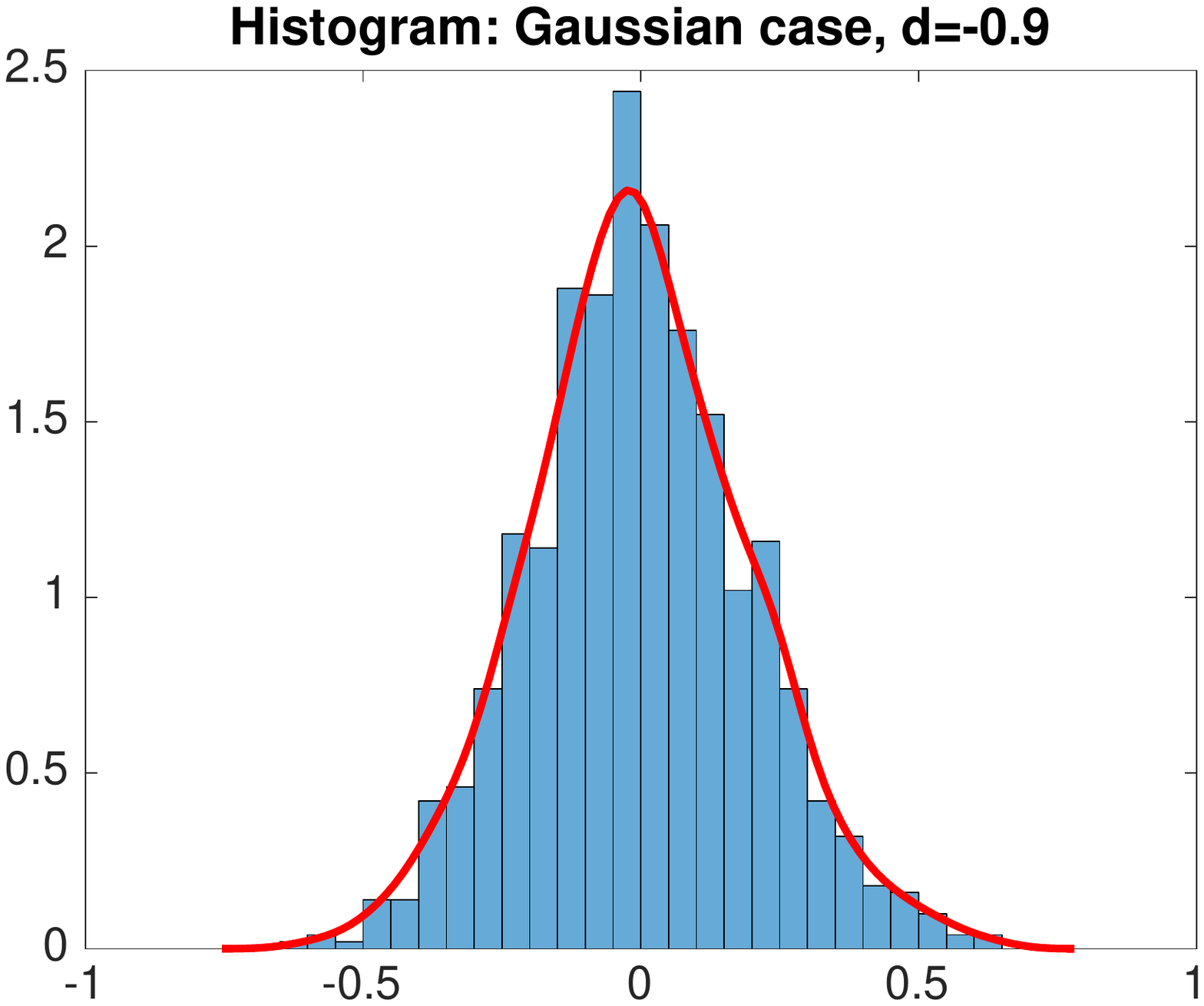}
\includegraphics[width=0.33\linewidth]{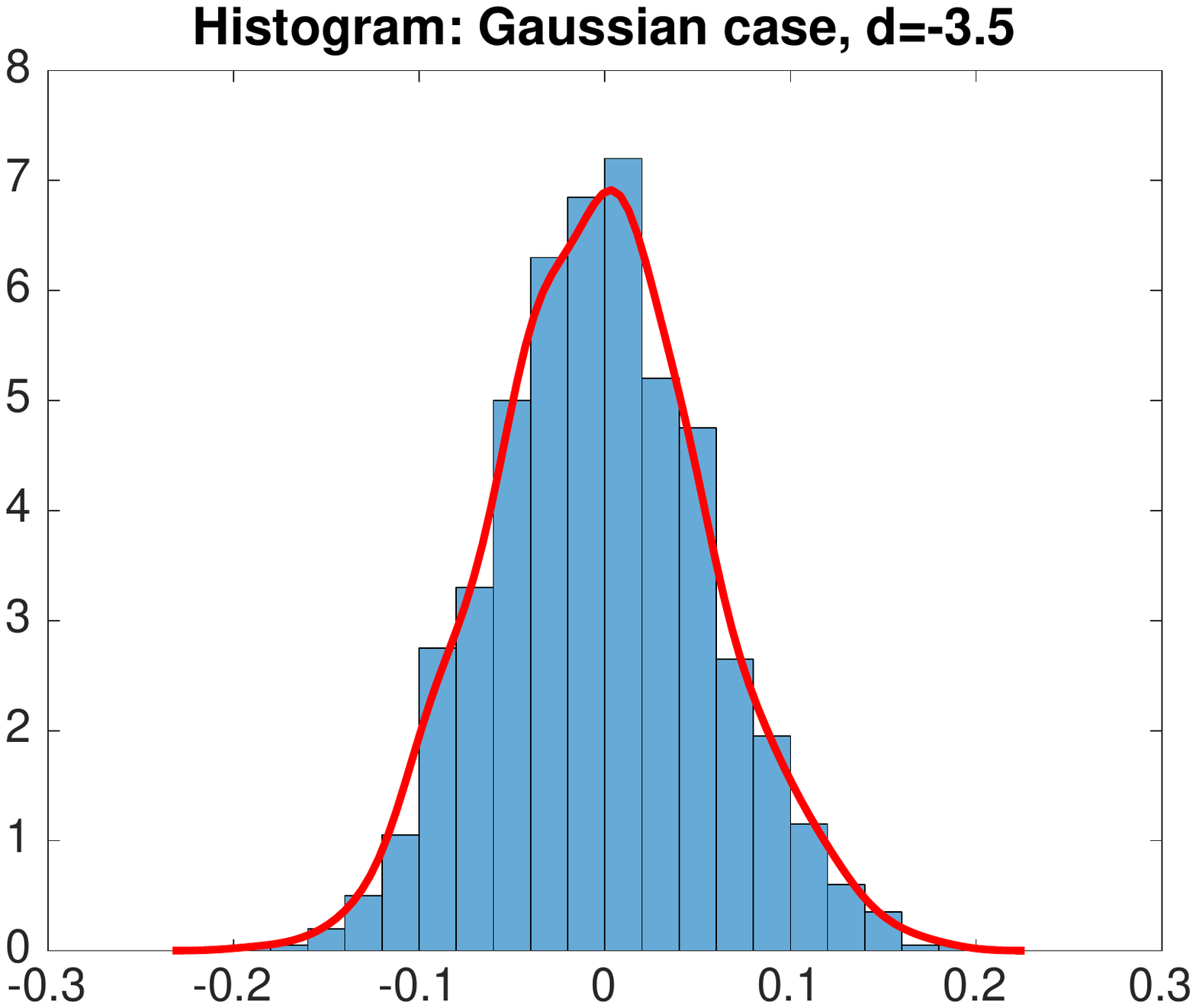}
\end{tabular}
\caption {The histograms with kernel density fits of $r_n$ from $m=1,000$ simulated processes with Gaussian innovations and $d=-0.1, -0.9, -3.5$. \label{fig:1}}
\end{figure}
%
%
\begin{figure}[H]
\centering
\begin{tabular}{cc}
\includegraphics[width=0.33\linewidth]{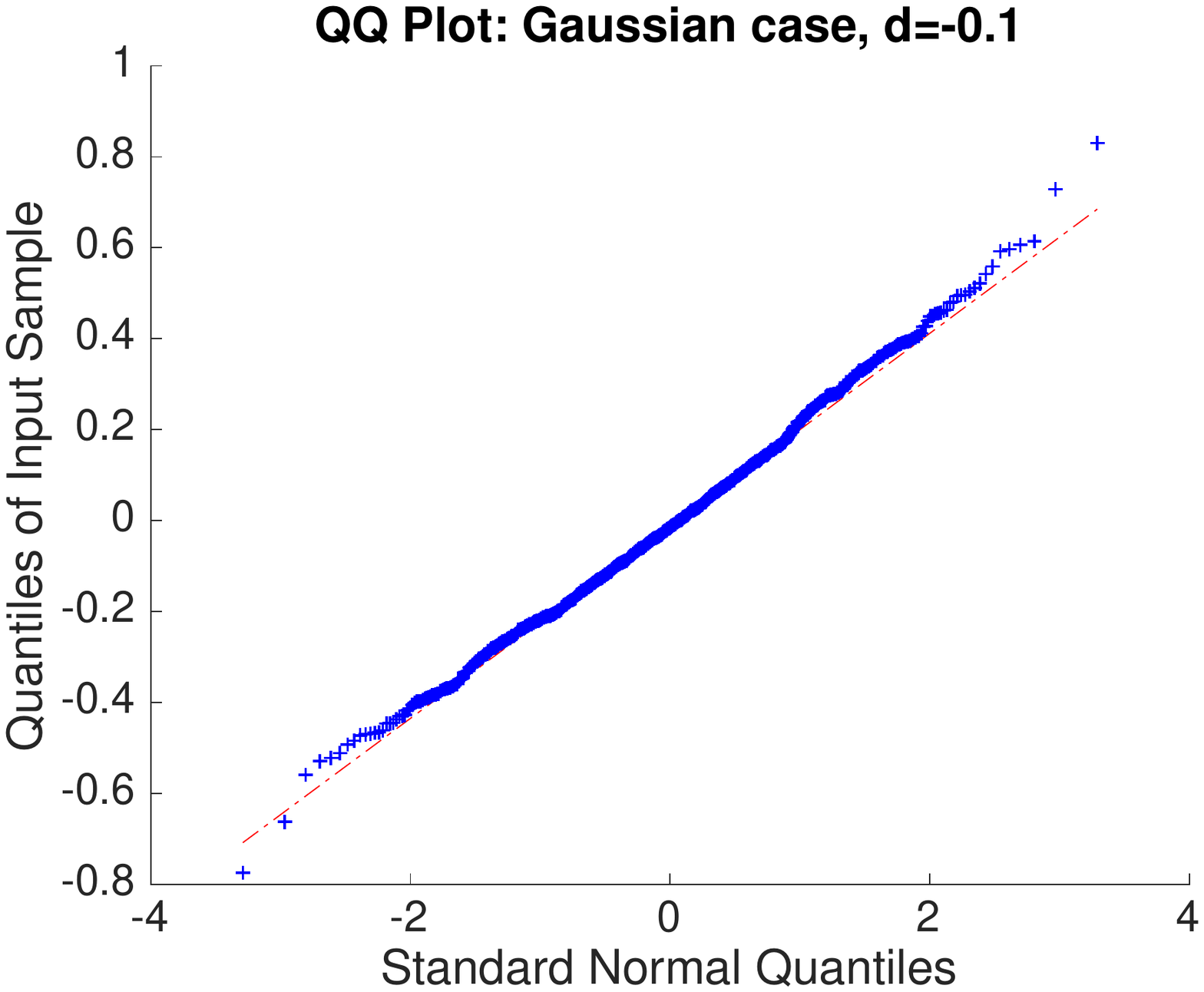}
\includegraphics[width=0.33\linewidth]{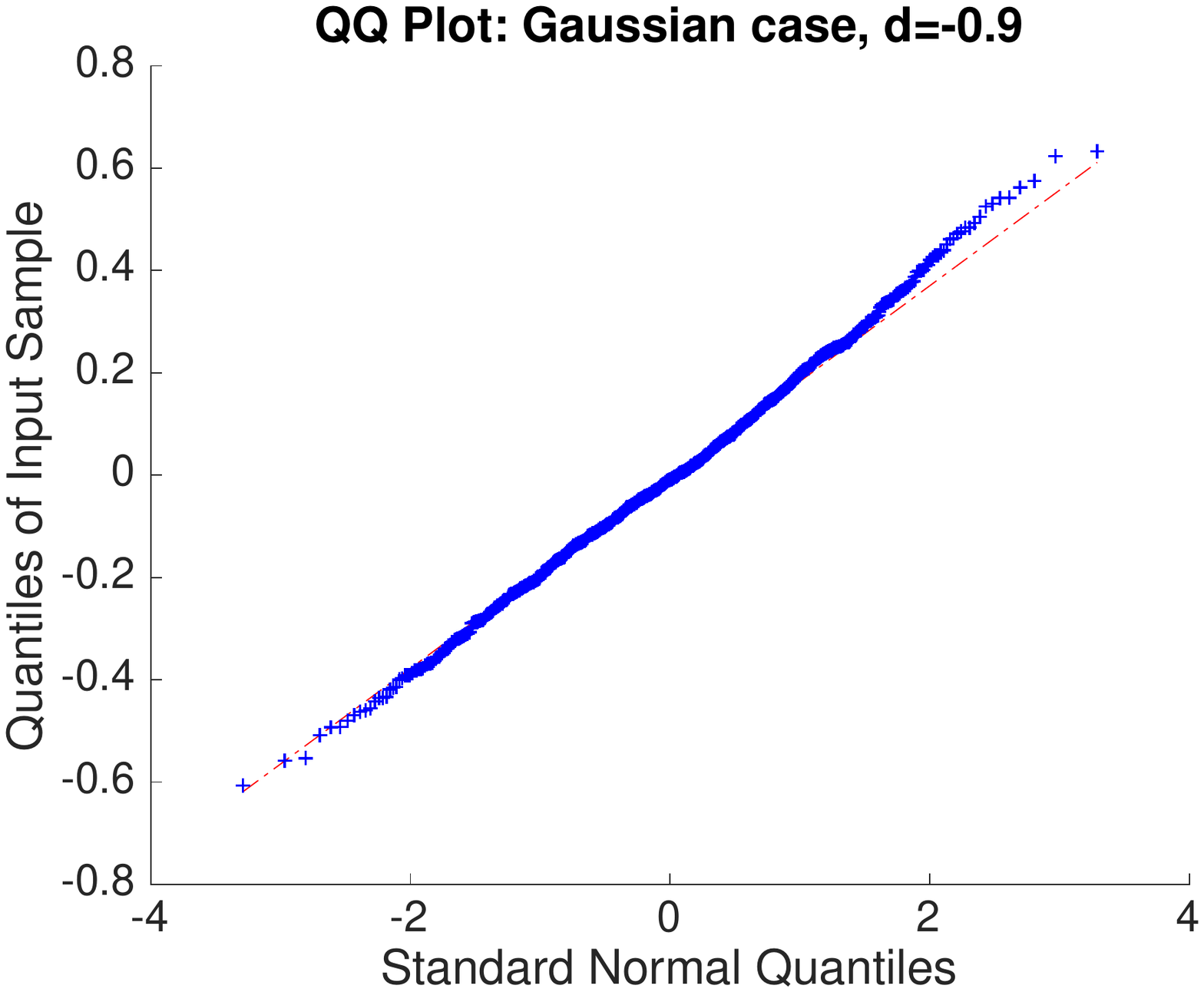}
\includegraphics[width=0.33\linewidth]{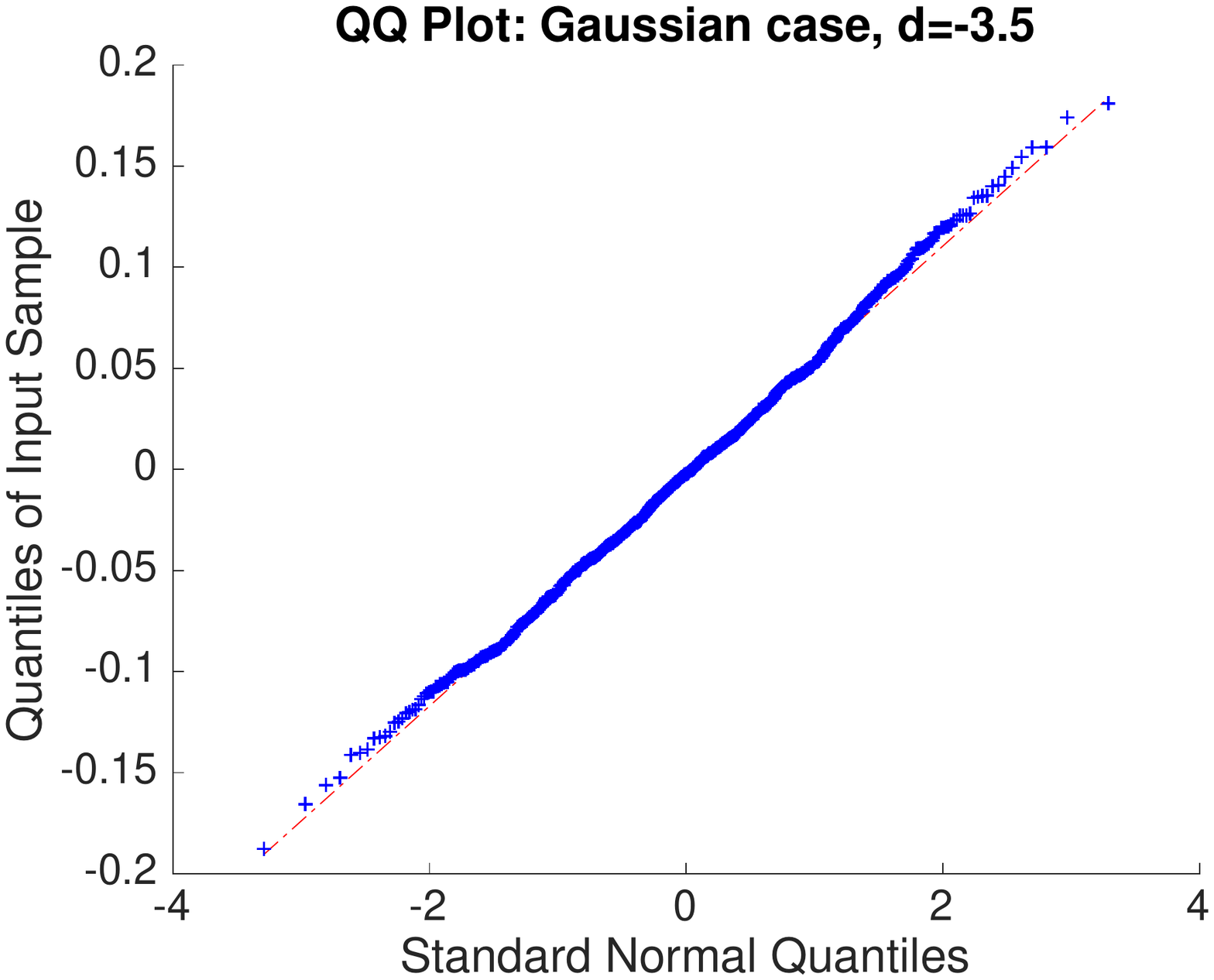}
\end{tabular}
\caption {The normal Q-Q plots of $r_n$ from $m=1,000$ simulated processes with Gaussian innovations and $d=-0.1, -0.9, -3.5$.\label{fig:2}}
\end{figure}
%

 \begin{table}[H]
\center
\caption{ The p-values on testing $H_0: \mu=0; H_1: \mu\ne 0$ for the asymptotic distribution of $r_n$ and the 95\% confidence intervals for $\mu$.} 
\label{tab3}
\bigskip
\begin{tabular}{| l | l | l | l | l |}
\hline                                                     
$d$                                          & $-0.1$  & $-0.3$ &  $-0.5$ &  $-0.7$ \\
\hline

p-value & 0.3996 & 0.7810 & 0.5058 & 0.6676  \\
\hline
CI & (-0.0191,0.0076) & (-0.0159,0.0120) & (-0.0086,0.0174) & (-0.0099, 0.0154)\\
\hline
\hline                                                     
$d$                                   &  $-0.9$   & $-1.5$  & $-2.5$ &  $-3.5 $ \\
\hline

p-value & 0.6705 & 0.8297 & 0.2055  &  0.2211    \\
\hline
CI & (-0.0148, 0.0095) & (-0.0064,0.0080) & (-0.0021, 0.0098)  & (-0.0058, 0.0013)    \\
\hline
\end{tabular}  
\end{table}
Figure \ref{fig:3} and Figure \ref{fig:4} show the histograms with kernel density fits and normal Q-Q plots of $r_n=\sqrt{n}\, \Big[T_n(h_n)-\int_{\R} f^2(x)\,dx\Big]$ from $m=1,000$ simulated processes with Cauchy innovations and $d=-0.1,-0.5, -0.9$.  It is clear that the realizations of $r_n$ do not follow normal distribution in the case $d=-0.1$ or $d=-0.5$. We also have performed simulation study for $d=-0.3, -0.7$ and obtained similar plots and results. For $d=-0.9$, the histogram and normal Q-Q plot indicate that the distribution of $r_n$ is very close to some normal distribution with mean $0$. This is because the central limit theorem (\ref{clt2}) in Corollary \ref{corollary2} holds for $\alpha=1+\epsilon$ with $\epsilon$ an arbitrary small positive number, $d<1-\frac{2}{\alpha}$, which is close to the case $\alpha=1$, $d=-0.9$ here.  In the case $0<\alpha\le 1$, we have the central limit theorem (\ref{r6}) although  (\ref{clt2}) does not hold any more. 
\begin{figure}[H]
\centering
\begin{tabular}{cc}
\includegraphics[width=0.33\linewidth]{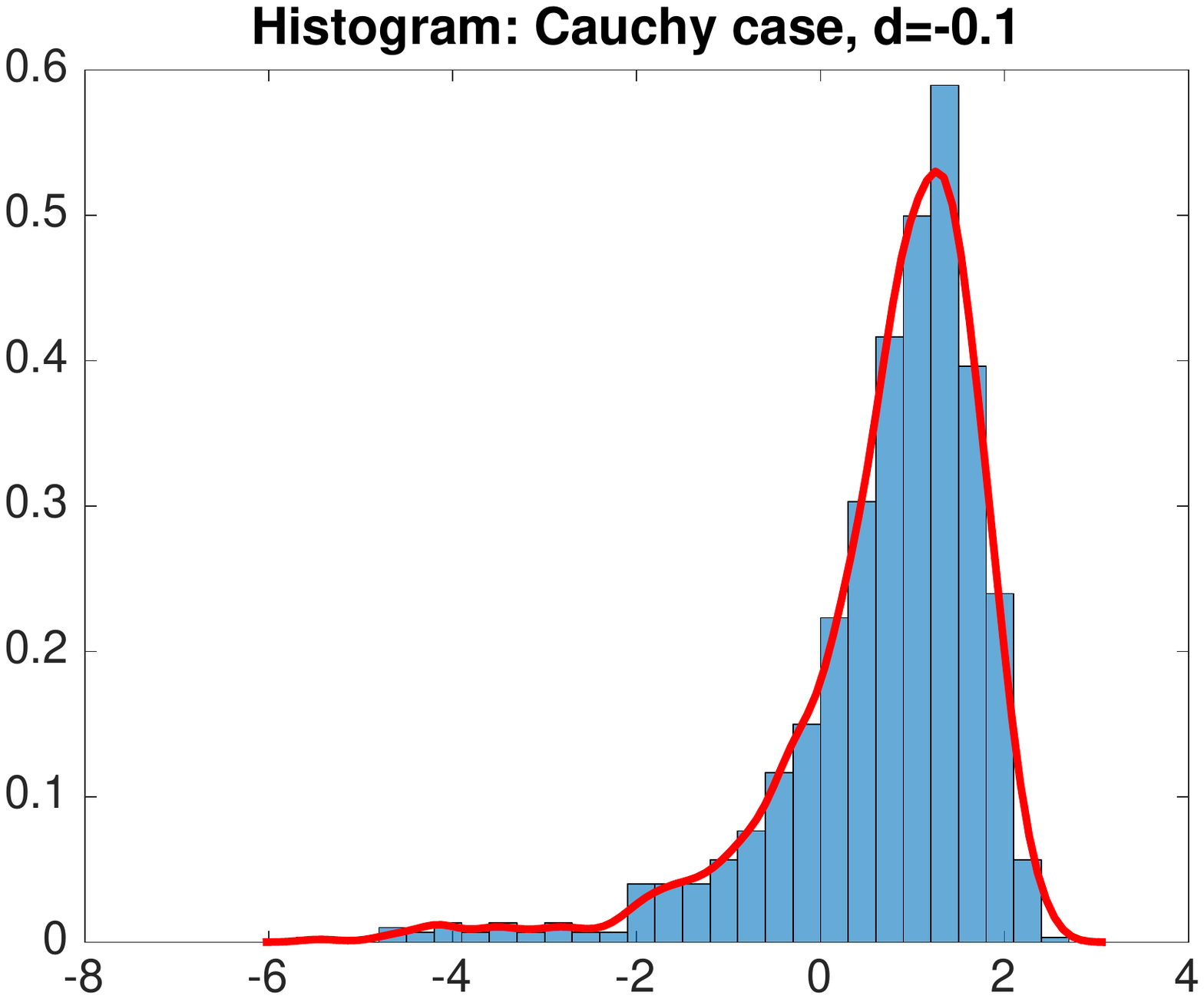}
\includegraphics[width=0.33\linewidth]{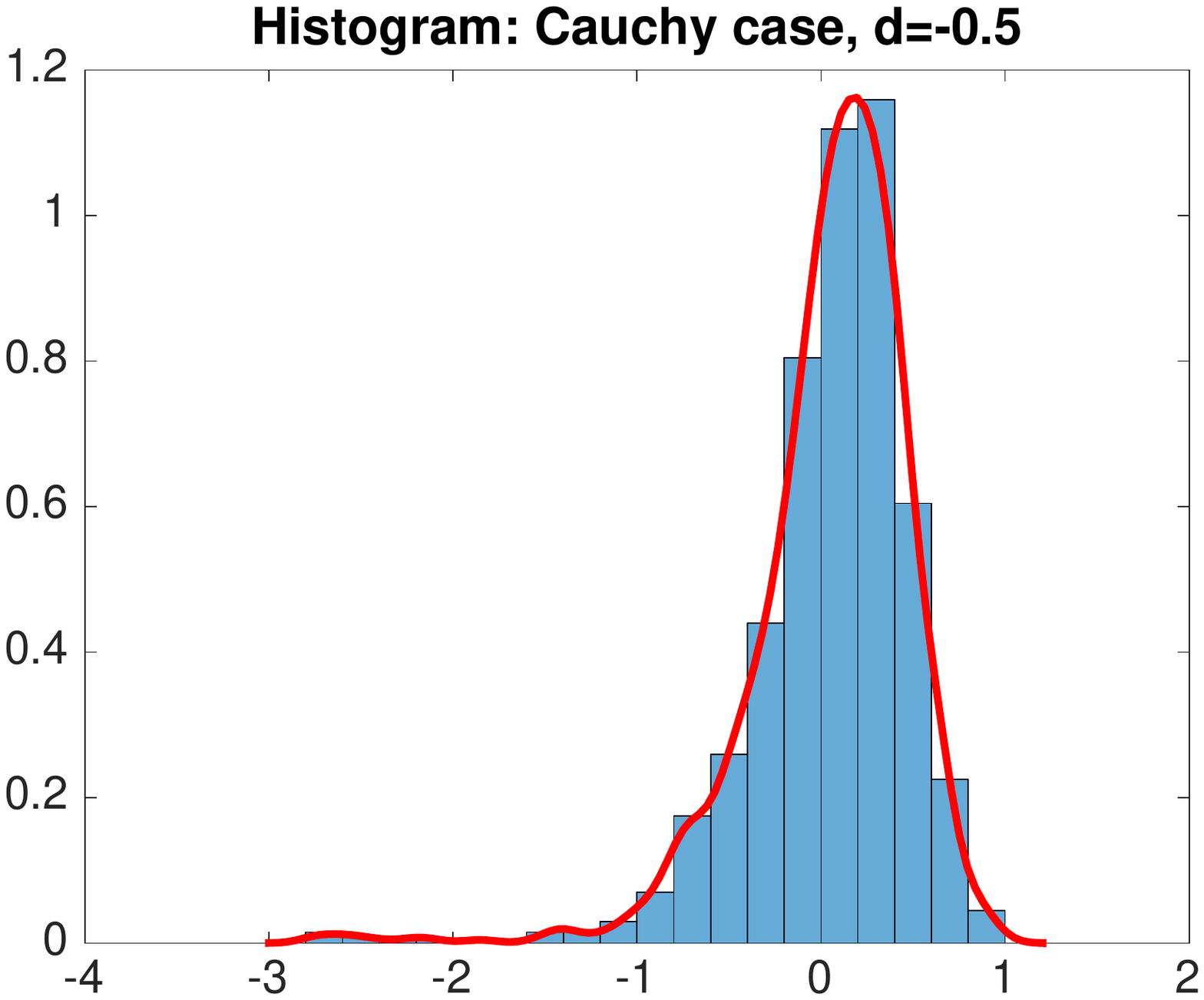}
\includegraphics[width=0.33\linewidth]{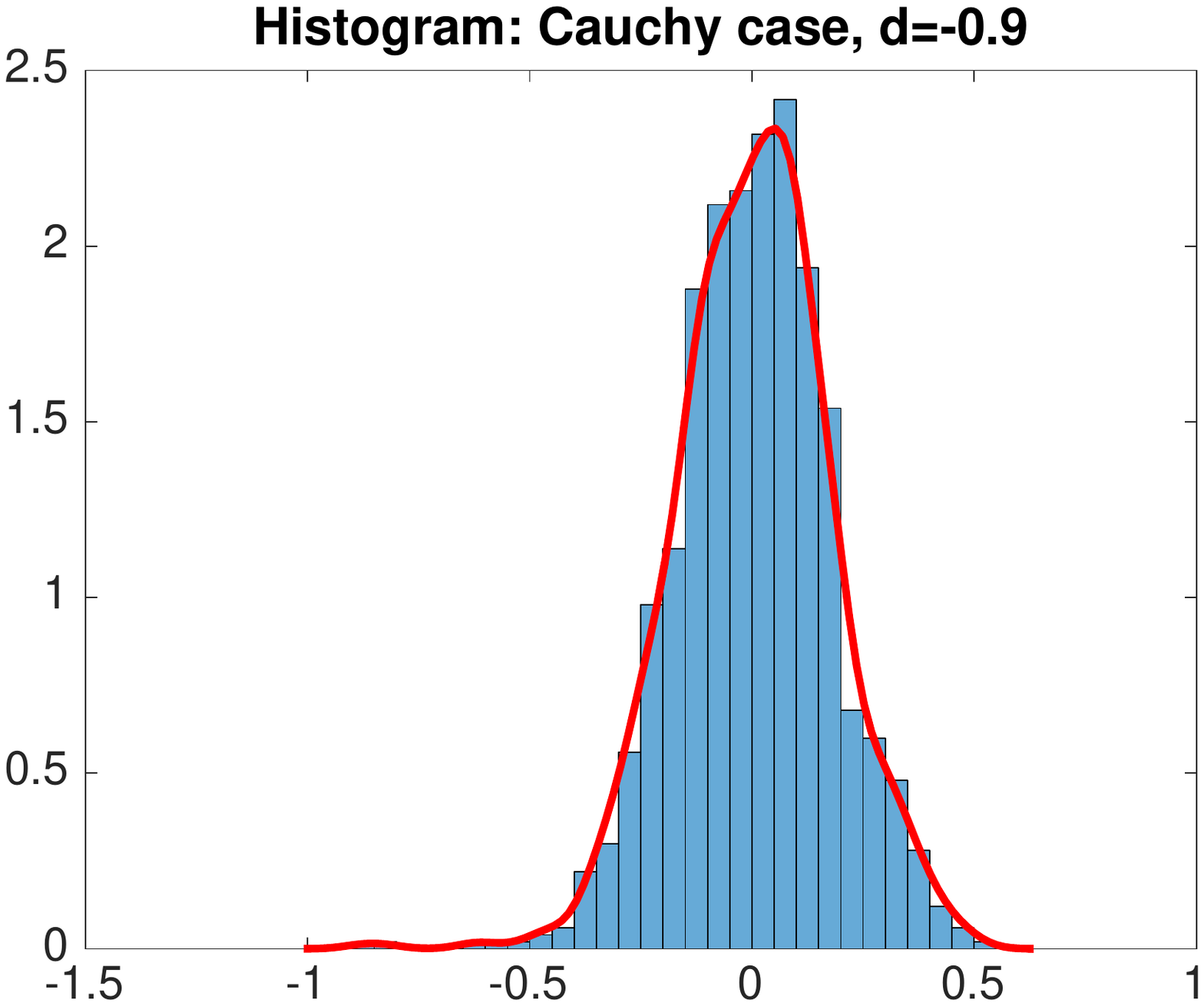}
\end{tabular}
\caption {The histograms with kernel density fits of $r_n$ from $m=1,000$ simulated processes with Cauchy innovations and $d=-0.1, -0.5, -0.9$.\label{fig:3}}
\end{figure}
\begin{figure}[H]
\centering
\begin{tabular}{cc}
\includegraphics[width=0.33\linewidth]{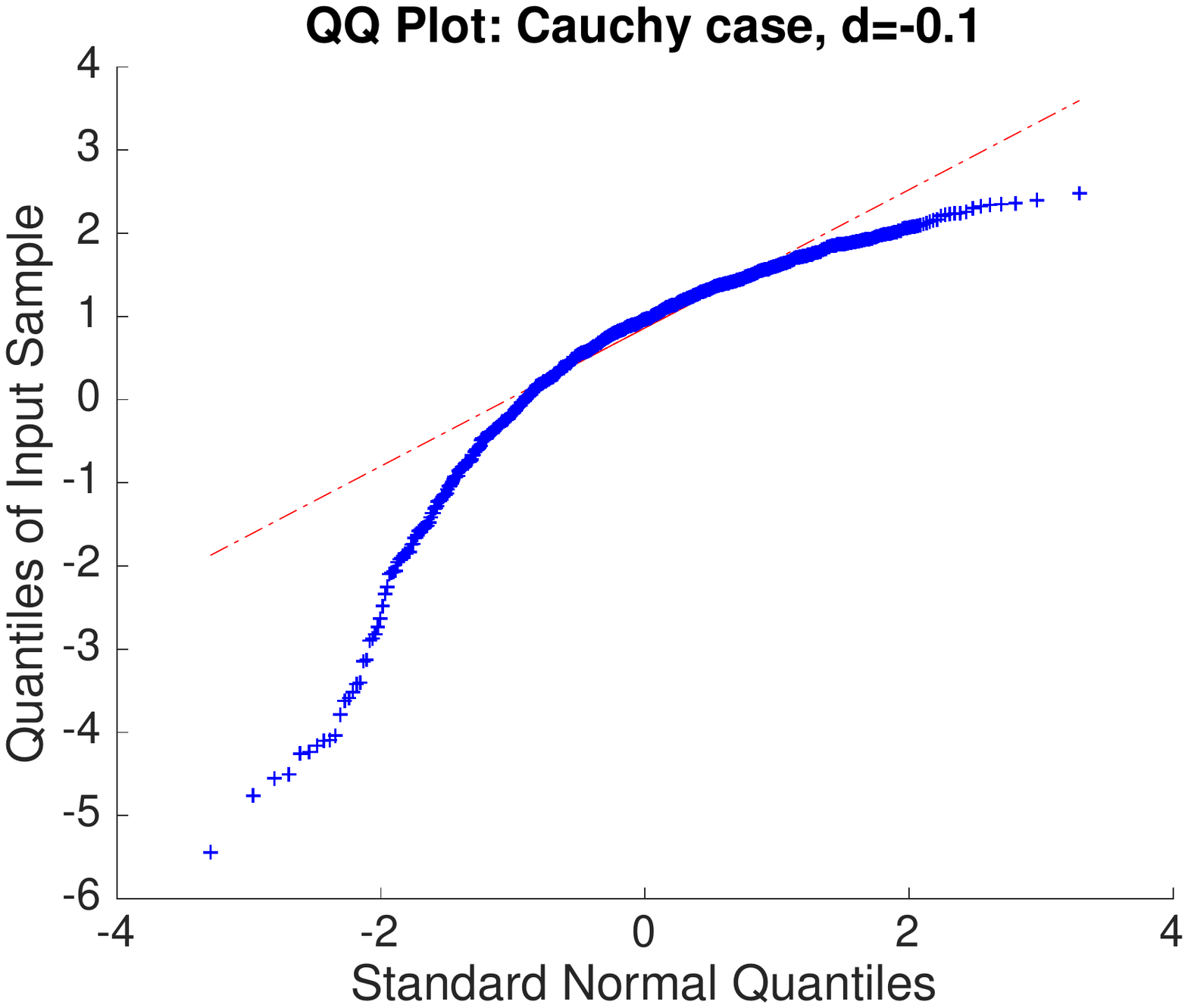}
\includegraphics[width=0.33\linewidth]{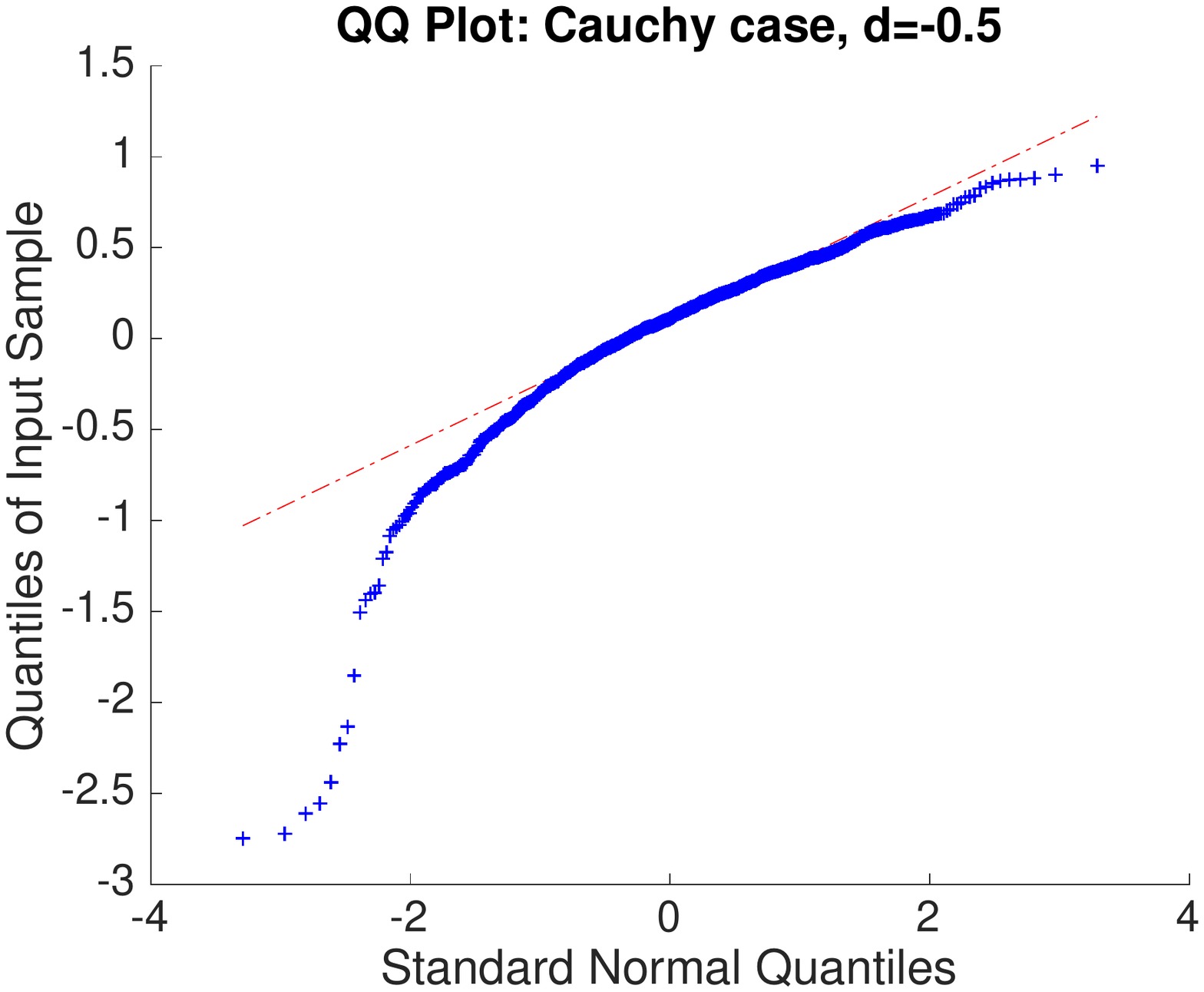}
\includegraphics[width=0.33\linewidth]{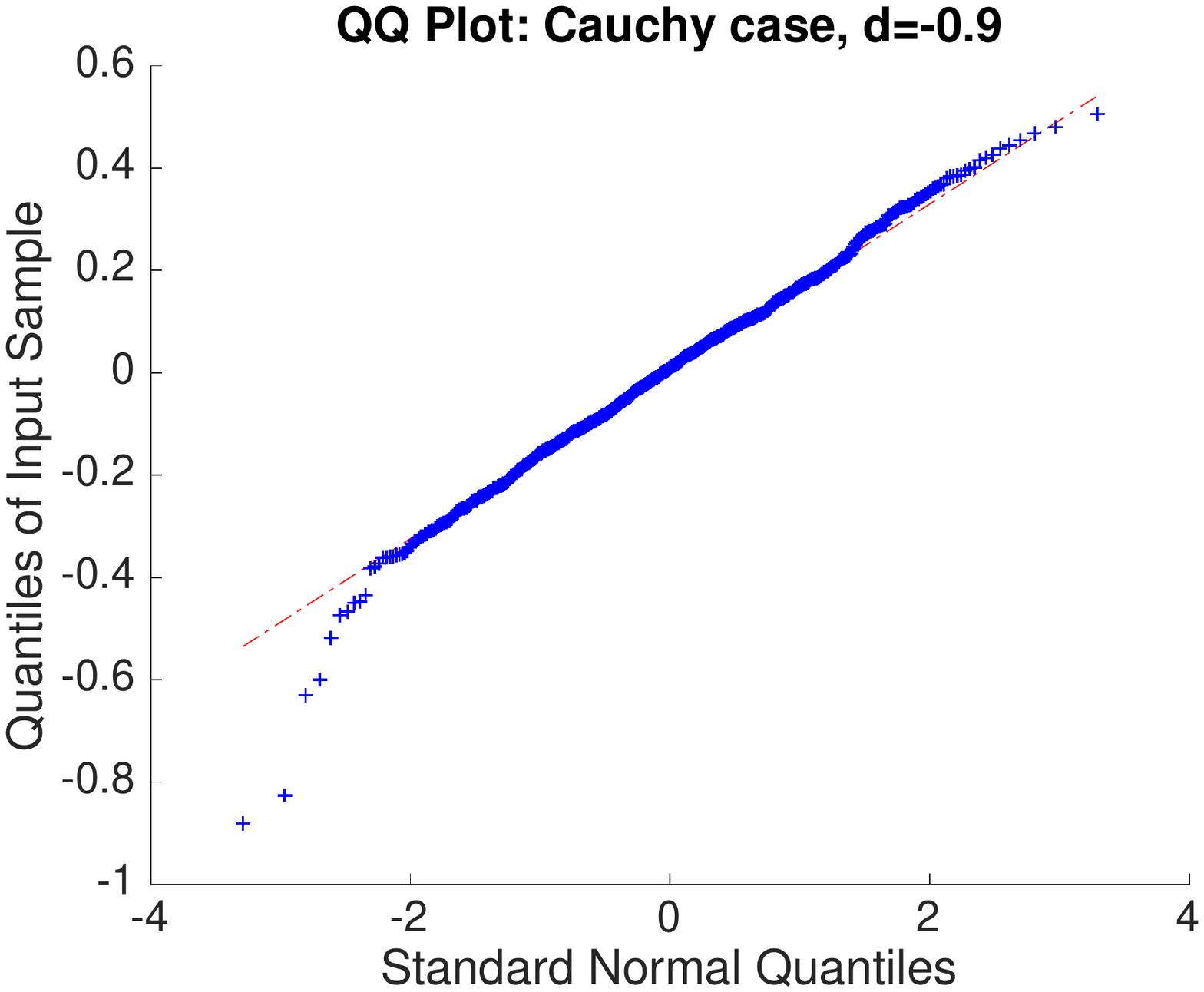}
\end{tabular}
\caption {The normal Q-Q plots of $r_n$ from $m=1,000$ simulated processes with Cauchy innovations and $d=-0.1, -0.5, -0.9$.\label{fig:4}}
\end{figure}

Figure \ref{fig:5} and Figure \ref{fig:6} show the histograms with kernel density fits and normal Q-Q plots of $\sqrt{n}T_n(h_n)$ from $m=1,000$ simulated processes with stable innovations ($\alpha=0.5$) and $d=-1.8,-2.5, -3.5$.  It is clear that the values of $\sqrt{n}T_n(h_n)$ do not follow normal distribution in the case $d=-1.8$ while they do follow normal distribution in the case $d=-3.5$. For $d=-2.5$, the histogram and normal Q-Q plot indicate that the distribution of $\sqrt{n}T_n(h_n)$ is not Gaussian but is close to some normal distribution. This confirms the analysis that  $\sqrt{n}T_n(h_n)$ follows a normal distribution when $d<1-\frac{2}{\alpha}$ even if $0<\alpha\le 1$.
\begin{figure}[H]
\centering
\begin{tabular}{cc}
\includegraphics[width=0.33\linewidth]{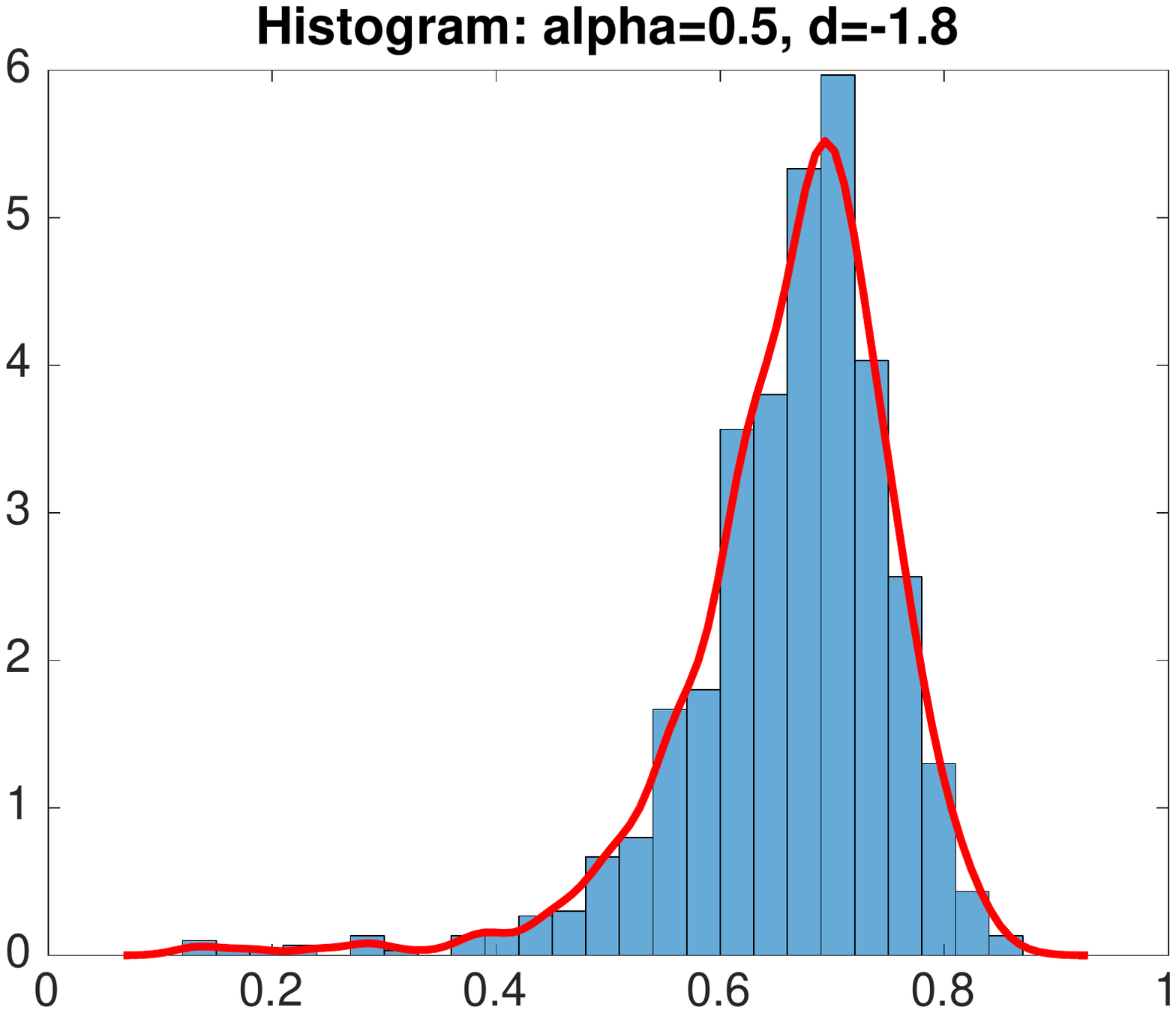}
\includegraphics[width=0.33\linewidth]{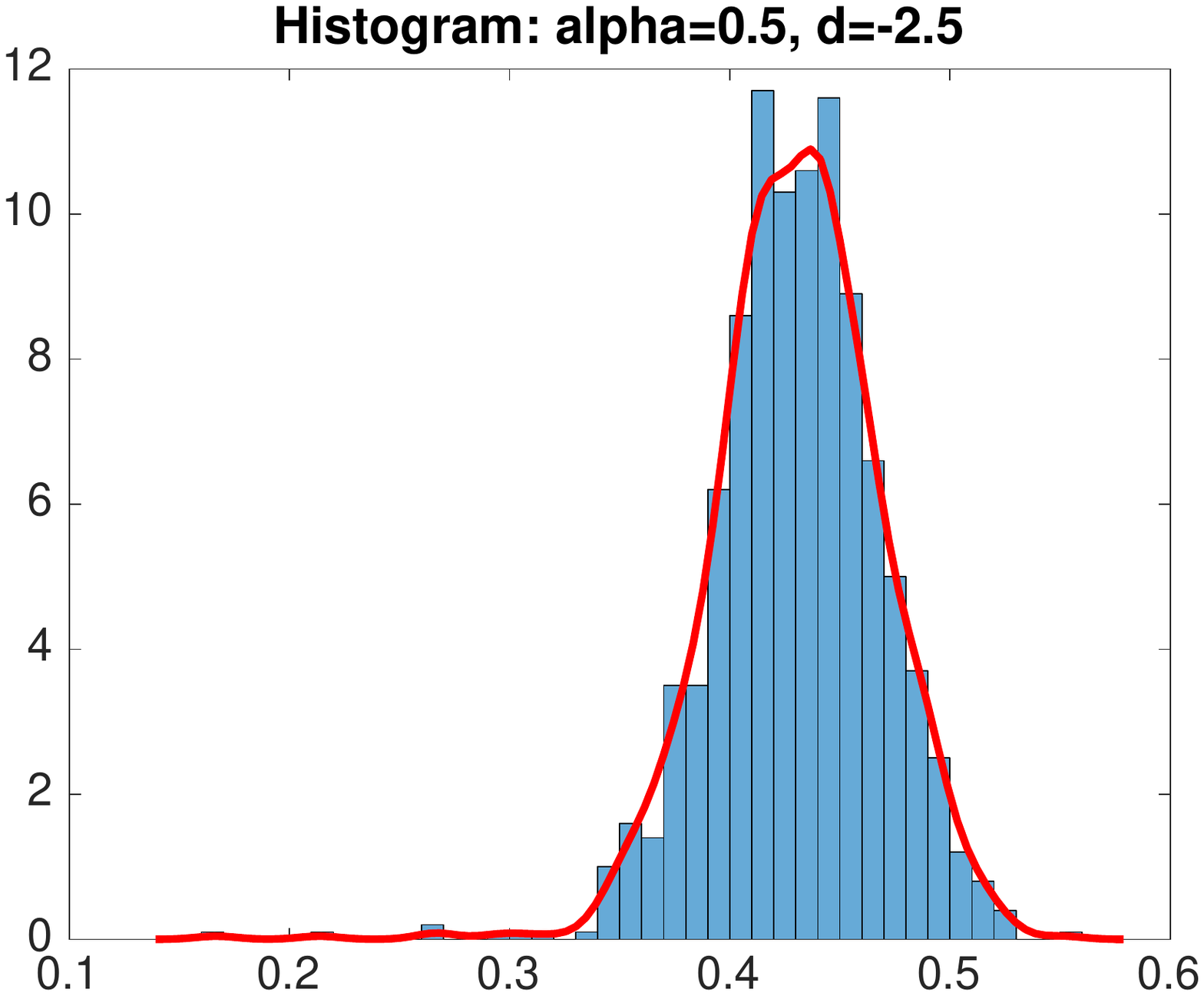}
\includegraphics[width=0.33\linewidth]{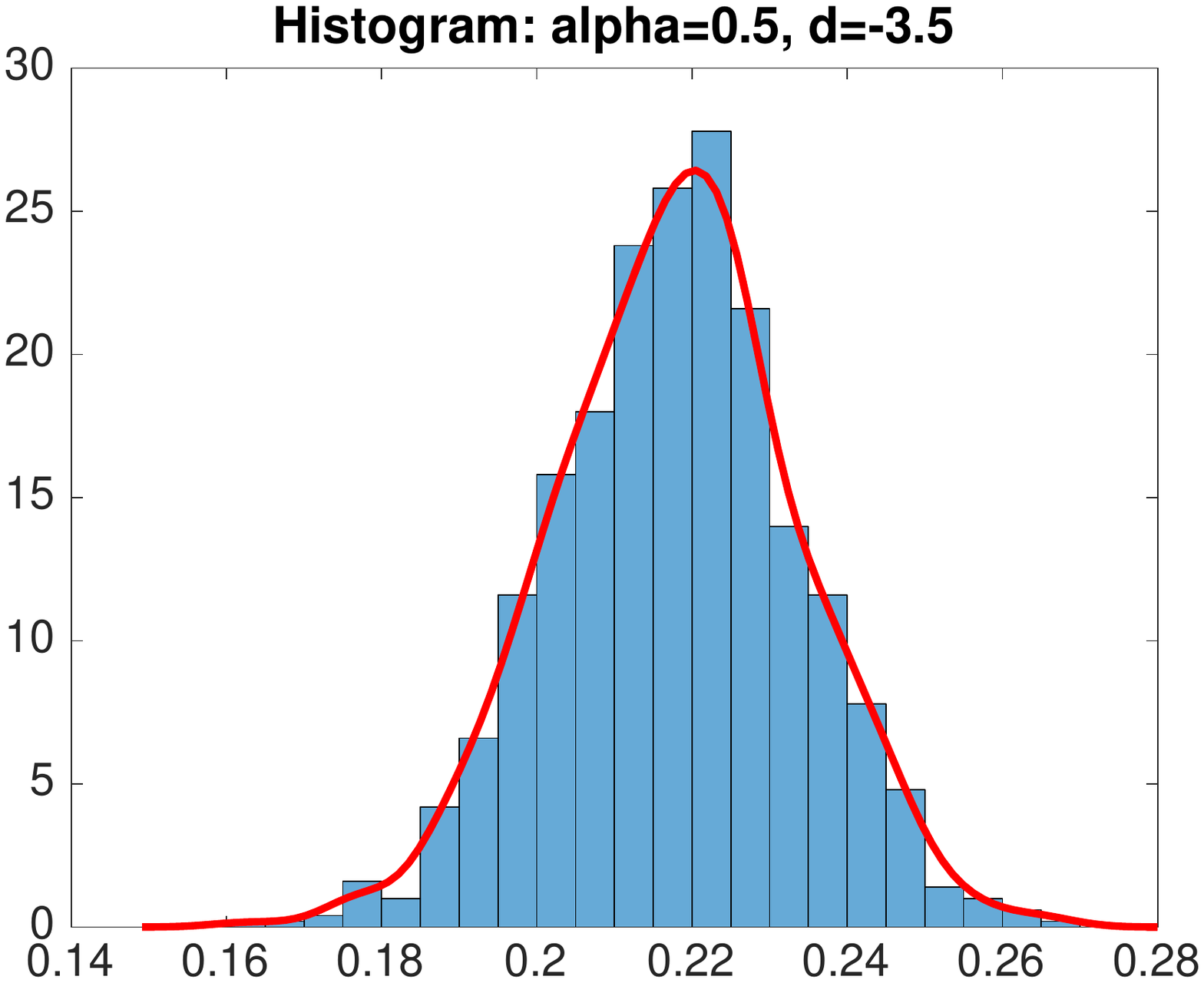}
\end{tabular}
\caption {The histograms with kernel density fits of $\sqrt{n}T_n(h_n)$ from $m=1,000$ simulated processes with stable innovations ($\alpha=0.5$)  and $d=-1.8, -2.5, -3.5$.\label{fig:5}}
\end{figure}
\begin{figure}[H]
\centering
\begin{tabular}{cc}
\includegraphics[width=0.33\linewidth]{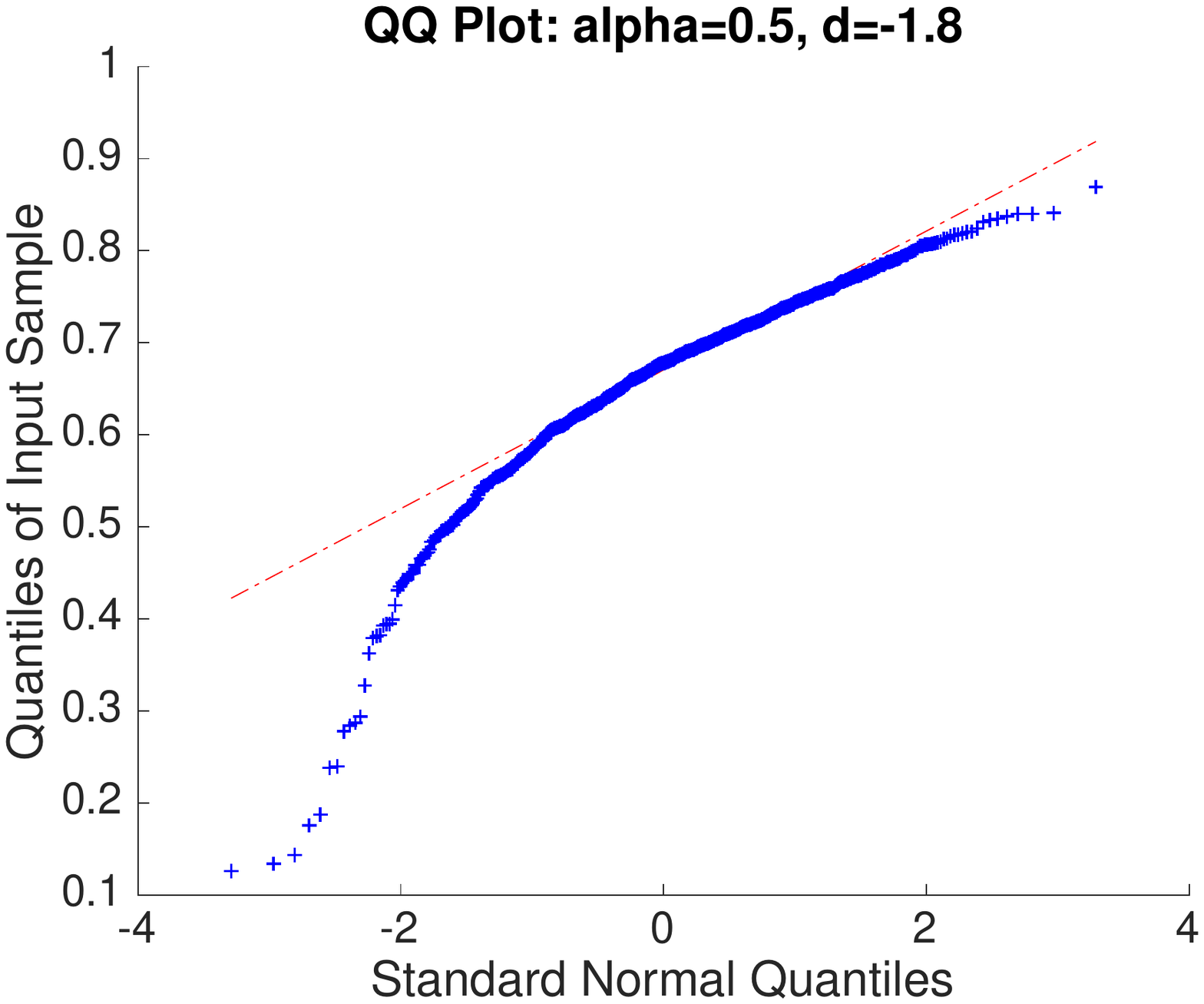}
\includegraphics[width=0.33\linewidth]{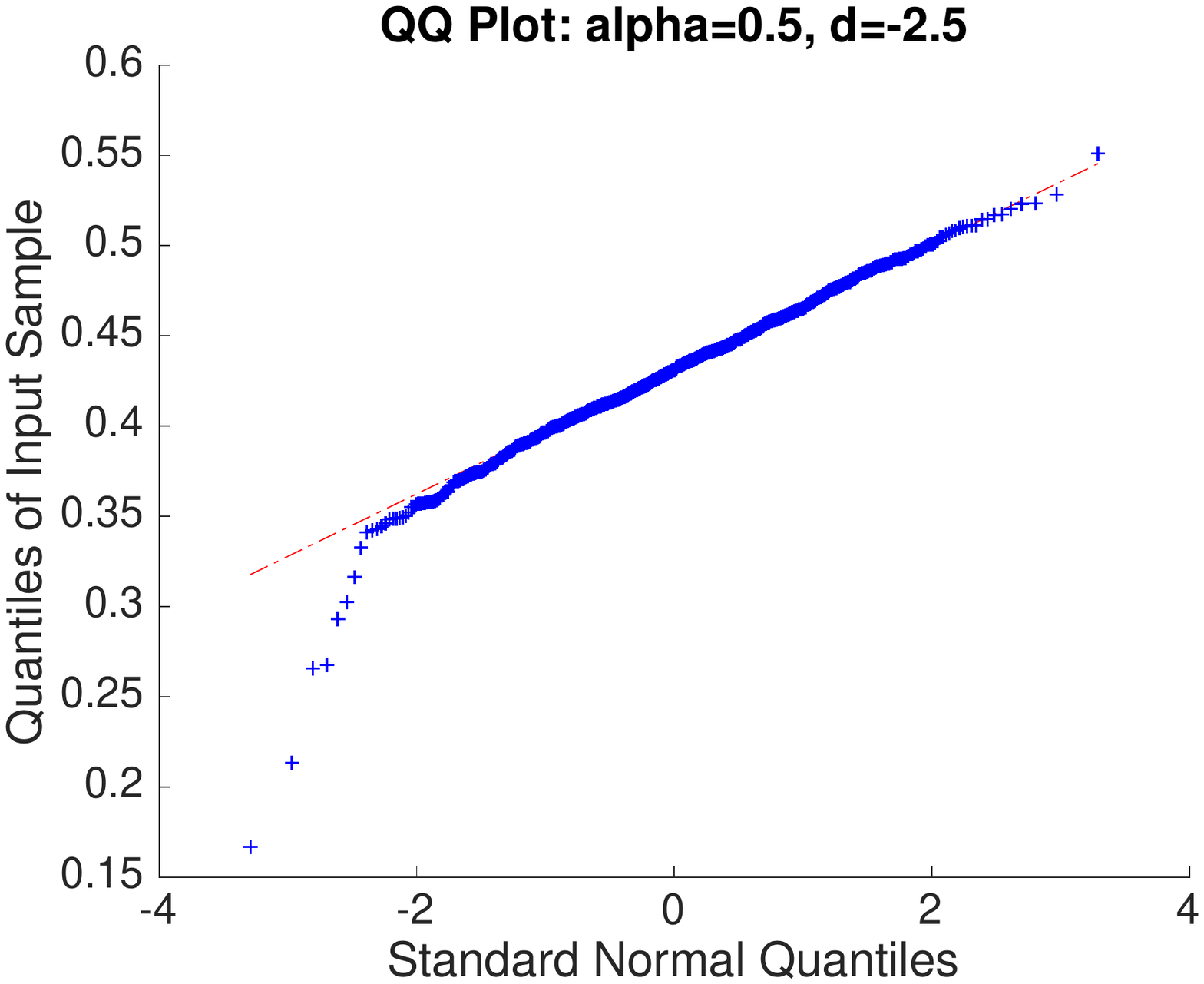}
\includegraphics[width=0.33\linewidth]{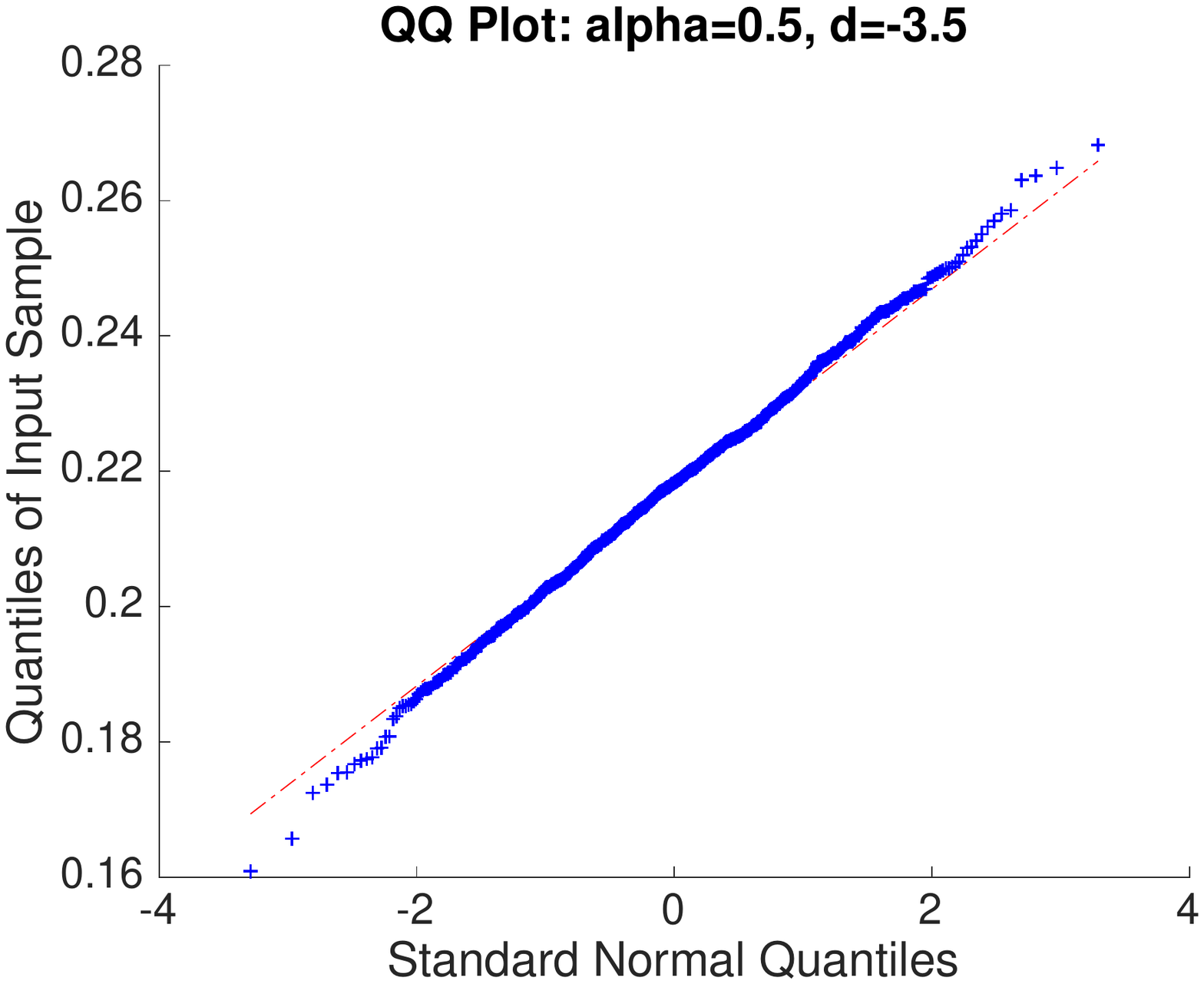}
\end{tabular}
\caption {The normal Q-Q plots of $\sqrt{n}T_n(h_n)$ from $m=1,000$ simulated processes with stable innovations ($\alpha=0.5$)  and $d=-1.8, -2.5, -3.5$. \label{fig:6}}
\end{figure}

For the purpose of illustration, we apply the $L^2_2$ kernel divergence estimator $\widehat{D}(f,g)$ to study the difference of annual average river flows among four rivers on the earth.  The data sets consist of annual average river flows (1,000 cubic meters per second) from  Danube River at Orshava, Romania, Gota River near Sjotop-Vannersburg, Sweden,  Mississippi River near St. Louis, Missouri, USA and Rhine River near Basle, Switzerland for 96 years between 1861 and 1956. The river flow data of these rivers are well studied in ecology and statistics, e.g., Hipel and McLeod (1994). 
The original data sets cover different time periods.  We choose part of each data set from year 1861 to year 1956 in order to compare the flows in the same time period.  
The Kwiatkowski, Phillips, Schmidt, and Shin (KPSS) test, the time series plots in Figure \ref{fig:7} and the autocorrelation plots in Figure \ref{fig:8} show the stationarity and short memory property of the river flow level of each river. 

We apply the estimator to calculate the divergence between the flow distributions of every two rivers. As in the simulation study, we take the standard normal density function as the kernel function. The right part of Figure \ref{fig:7} gives the kernel density estimations of river flows for each river.  Since they  approximately follow some normal distributions, we let the bandwidth to be $h_n=96^{-0.4} \approx 0.161$. All the divergences are presented in Table \ref{tab:divergence}.  According to the density curves in Figure \ref{fig:7}, it is not a surprise that the biggest $L^2_2$ divergence is between the Gota River and the Rhine River. The density functions of these two rivers do not have much overlap although they are very close in location. On the other hand, the density functions of  Gota River and  Rhine River, in particular the density function of the Gota River,  have much bigger values in their supports than those of the other two rivers. The squared difference $(f(x)-g(x))^2$, not the relative location between the supports of $f(x)$ and $g(x)$, makes contribution to the value of the $L^2_2$ divergence $D(f,g)$. Mississippi River and Danube River have very similar density functions in terms of supports and values. Mississippi River has the smallest maximum density and it has a little bit overlap with Gota River or  Rhine River. There is almost no overlap between the Danube River and Gota River or  Rhine River. This explains the order of the divergences in Table \ref{tab:divergence}.

\begin{figure}[H]
\centering
\begin{tabular}{cc}
\includegraphics[width=0.45\linewidth]{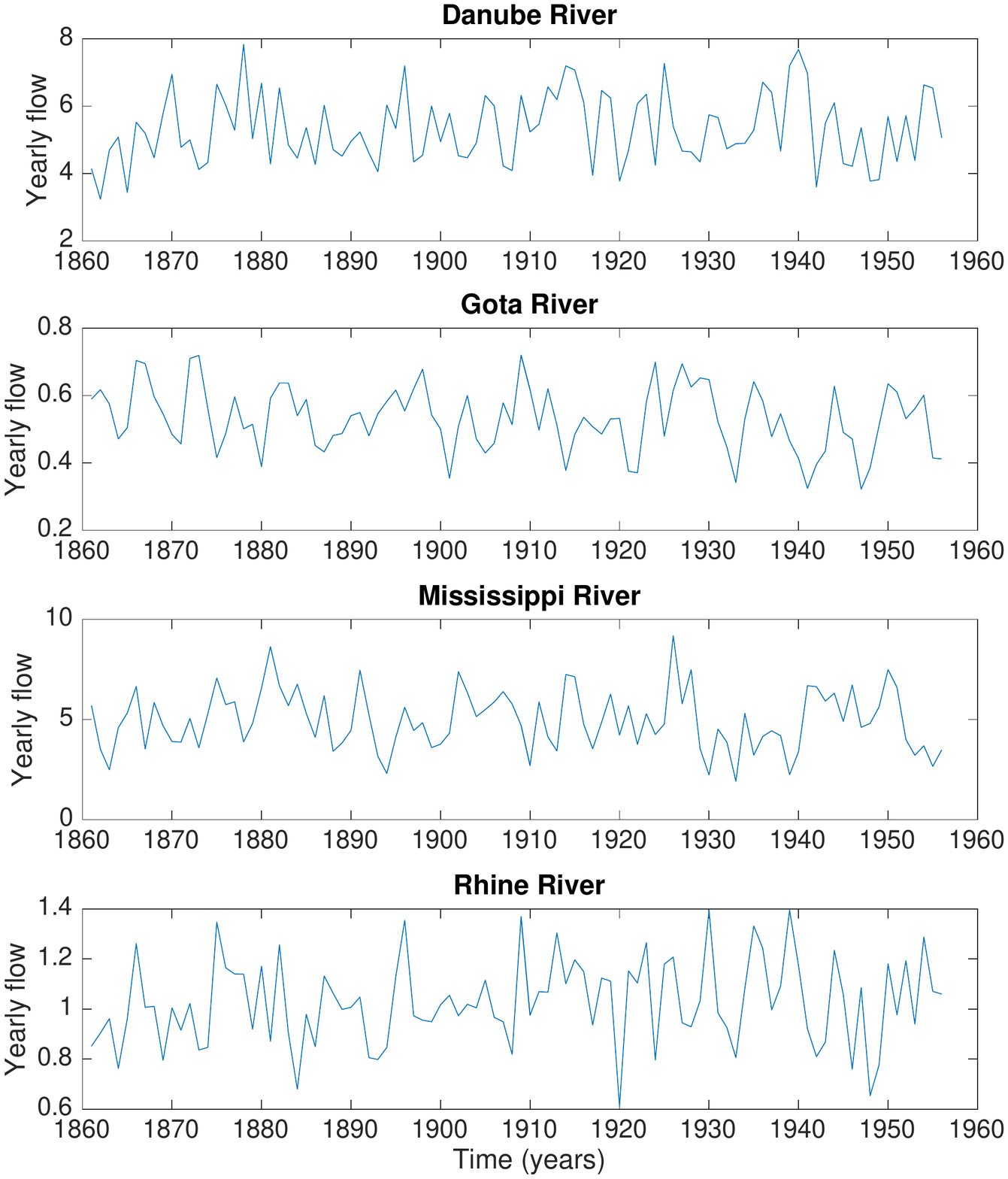}
\includegraphics[width=0.45\linewidth]{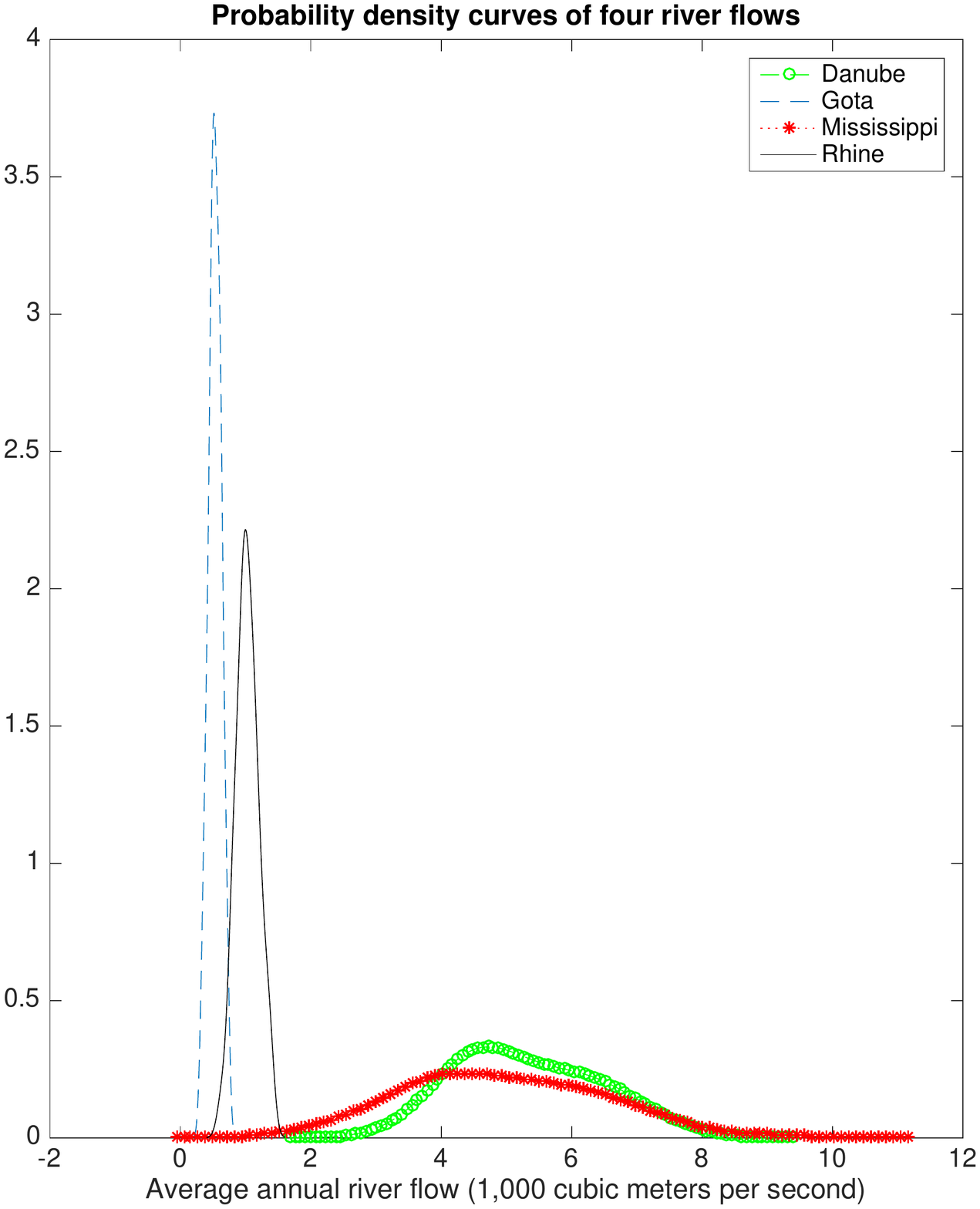}
\end{tabular}
\caption {The time series plots and probability density functions of average annual river flow. \label{fig:7}}
\end{figure}

\begin{figure}[H]
\centering
\begin{tabular}{cc}
\includegraphics[width=0.55\linewidth]{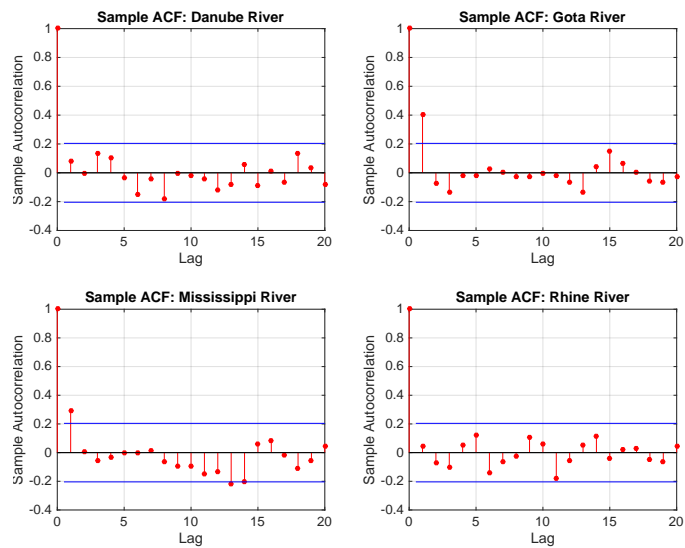}
\end{tabular}
\caption {The sample autocorrelation functions of average annual river flow. \label{fig:8}}
\end{figure}

 \begin{table}[H]
\center
\caption{River flow divergences} 
\label{tab:divergence}
\bigskip
\begin{tabular}{| l | l |}
\hline                                                     
\;\;\;\;\;\;\;\;\;\;\;\;\;\;\;\;\;\;The rivers                                         &Divergence   \\
\hline

\;\;Gota  River,  \;\;\; Rhine River &\;\;\; 2.7797  \\
\hline
\;\;Gota River, \;\;\; Danube River &\;\;\; 2.1421  \\
\hline
\;\;Gota River, \;\;\; Mississippi River &\;\;\; 2.0725 \\
\hline
\;\;Rhine River, \;\;\;Danube River &\;\;\; 1.6129  \\
\hline
\;\;Rhine River, \;\;\;Mississippi River &\;\;\; 1.5433  \\
\hline
\;\;Danube River, Mississippi River &\;\;\; 0.0348 \\
\hline
\end{tabular}  
\end{table}

\section{Extensions to multivariate linear processes}\label{extension}

In this section, we denote norm functions in various linear spaces by $|\cdot|$. For $x\in\R^d$, $|x|:=\sqrt{x\, x^T}$ where $x^T$ is the transposition of $x$. For a $d\times d$ deterministic real matrix $a$, $|a|:=\max\limits_{x\in\R^d, |x|=1} |a\, x^T|$. In the sequel, we give extensions to the following multivariate linear process
\begin{align}\label{mlp}
X_n=\sum\limits_{i=0}^\infty a_i\, \varepsilon_{n-i},
\end{align}
where the innovations $\varepsilon_i$ are i.i.d. $d\times 1$ random vectors in some probability space $(\Omega, \mathcal{F}, \P)$ and $a_i$ are $d\times d$ deterministic real matrices such that $\sum\limits_{i=0}^\infty a_i\, \varepsilon_{n-i}$ converges in distribution. It is easy to see that the multivariate linear process $X_n$ in (\ref{mlp}) exists and is well defined when $\varepsilon_i$ are i.i.d. $d\times 1$ random vectors in $L^p(\R^d)$ for some $p>0$ and $a_i$ are $d\times d$ deterministic real matrices satisfying $\sum\limits^{\infty}_{i=0}|a_i|^{2\wedge p}<\infty$.  Let $f$ be the probability density function of $X_n$. To estimate the quadratic functional of $f$, we would use the estimator
\[ 
T_n(h_n)=\frac{1}{n(n-1)\det(h_n)} \sum_{1\le i\neq j\le n}K\left(h^{-1}_n(X_i-X_j)\right),
\]
where the kernel $K$ is a symmetric and bounded function with $\int_{\R^d} K(u)\, du = 1$ and $\int_{\R^d} |u|^2|K(u)|\, du<\infty$. The bandwidth sequence $h_n$ are diagonal matrices $diag[h_{n1}, h_{n2}, \cdots, h_{nd}]$ satisfying $0<h_{ni}\to 0$ for all $i=1,2,\cdots, d$.  

We require the following assumptions on $a_i$ and $\varepsilon_i$.
\begin{assumption} \label{ass1}
$\det(a_0)\neq 0$, $\det(a_m)\neq 0$ for some $m\geq 1$ and $a_i=d\times d\; \text{zero matrix}$ for all $i>m$, and  $\int_{\R^d}|\lambda|^{2\gamma} |\phi_{\varepsilon}(\lambda)|^2\, d\lambda<\infty$ for some $\gamma\in(0,1]$.
\end{assumption}

\begin{assumption} \label{ass2}
$\det(a_0)\neq 0$, $\det(a_p)\neq 0$ and $\det(a_q)\neq 0$ with $p=\min\{i\in\N: \det(a_i)\neq 0\}$ and $q=\min\{i\in\N: \det(a_i)\neq 0\; \text{and}\; i>p\}<\infty$, $\sum\limits^{\infty}_{i=0} |a_i|^{\gamma}<\infty$, $\E|e^{\iota \lambda \varepsilon_1}-\phi_{\varepsilon}(\lambda)|^{4}\leq c_{\gamma,4} \left(|\lambda|^{4\gamma}\wedge 1\right)$ and $\int_{\R^d}|\lambda|^{2\gamma} |\phi_{\varepsilon}(\lambda)|^2\, d\lambda<\infty$ for some $\gamma\in(0,1]$, $\exists$ two indices $i_r$ s.t. $\det(a_{i_r+r}-a_{i_r})\neq 0$ for each $r\in\{1,2,\cdots, q\}$.
\end{assumption}

\begin{assumption} \label{ass3}
$\det(a_0)\neq 0$, $\det(a_p)\neq 0$ with $p=\min\{i\in\N: \det(a_i)\neq 0\}<\infty$, $\exists$ indice $i_r$ such that $\det(a_{i_r}-a_{i_r+r})\neq 0$ for each $r\in\{1,\cdots, p\}$, $\sum\limits^{\infty}_{i=0} |a_i|^{\gamma}<\infty$, $\E|e^{\iota \lambda \varepsilon_1}-\phi_{\varepsilon}(\lambda)|^{2}\leq c_{\gamma,2} \left(|\lambda|^{2\gamma}\wedge 1\right)$ and $\int_{\R^d}|\lambda|^{2\gamma} |\phi_{\varepsilon}(\lambda)|^2\, d\lambda<\infty$ for some $\gamma\in(0,1]$.
\end{assumption}

Using similar arguments as in the proof of Theorems \ref{thm1} and \ref{thm2} with some proper modifications, we can get the following results.

\begin{theorem} \label{thm5}  Under the assumption \ref{ass1} or \ref{ass2}, we further assume that $f$ is bounded. Then there exist positive constants $c_1$ and $c_2$ such that
\begin{align} \label{r71}
\Big|\E T_n(h_n)-\int_{\R^d} f^2(x)\, dx \Big|\leq c_1\left(\frac{1}{n}+\det(h_n)^{2\gamma}\right),
\end{align}
\begin{align}  \label{r72}
\E\Big|T_n(h_n)-\E T_n(h_n)-\frac{1}{n}\sum^n_{i=1}Y_i\Big|^2 \leq c_2 \Big(\frac{1}{n^2\det(h_n)}+\frac{\eta_{n,\gamma}}{n^2}+\frac{\det(h_n)^{2\gamma}}{n}\Big)
\end{align}
and, if additionally $n\det(h_n)\to \infty$ as $n\to\infty$,
\begin{align} \label{r73}
\sqrt{n}\, \Big[T_n(h_n)-\E T_n(h_n)\Big]\overset{\mathcal{L}}{\longrightarrow} N(0,4\sigma^2),
\end{align}
where $Y_i=2\big(f(X_i)-\int_{\R^d} f^2(x)\, dx\big)$, $\eta_{n,\gamma}=\sum\limits^{n}_{\ell=0}\sum\limits^{\infty}_{i=\ell} |a_i|^{\gamma}$
and $\sigma^2=\lim\limits_{n\to\infty} n^{-1}\Var(S_n)$ with 
\[
S_n=\sum\limits^n_{i=1}\left(f(X_i)-\int_{\R^d} f^2(x)\, dx\right).
\]
\end{theorem}

\bigskip

\begin{theorem} \label{thm6} 
Under the assumption \ref{ass3}, we further assume that $f$ is bounded and $\int_{\R^d} |\widehat{K}(\lambda)|\, d\lambda<\infty$. Then there exist positive constants $c_3$ and $c_4$ such that
\begin{align} \label{r74}
\Big|\E T_n(h_n)-\int_{\R^d} f^2(x)\, dx \Big|\leq c_3\left(\frac{1}{n}+\det(h_n)^{2\gamma}\right),
\end{align}
\begin{align}  \label{r75}
\E\Big|T_n(h_n)-\E T_n(h_n)-\frac{1}{n}\sum^n_{i=1}Y_i\Big| \leq c_4 \Big(\frac{1}{n\det(h_n)}+\frac{\det(h_n)^{\gamma}}{\sqrt{n}}\Big)
\end{align}
and, if additionally $\sqrt{n}\det(h_n)\to \infty$ as $n\to\infty$,
\begin{align} \label{r76}
\sqrt{n}\, \Big[T_n(h_n)-\E T_n(h_n)\Big]\overset{\mathcal{L}}{\longrightarrow} N(0,4\sigma^2)
\end{align}
for some $\sigma^2\in(0,\infty)$.

\end{theorem}

\bigskip

\begin{remark}
In the i.i.d case, that is, $\det(a_0)\neq 0$ and $a_i=d\times d\; \text{zero matrix}$ for all $i\geq 1$, and  $\int_{\R^d}|\lambda|^{2\gamma} |\phi_{\varepsilon}(\lambda)|^2\, d\lambda<\infty$ for some $\gamma\in(0,1]$, then results in Theorem \ref{thm5} still hold with the right hand side of (\ref{r71}) replaced by a constant multiple of $\det(h_n)^{2\gamma}$ and $\eta_{n,\gamma}$ in (\ref{r72}) by $0$, respectively.
\end{remark}

\bigskip

\section{Proofs}\label{proof}

In this section, the proofs of our main results, Theorems \ref{thm1} and \ref{thm2}, are given. First we provide two useful lemmas based on the projection method. 

For each $i\in\Z$, let  $\mathcal{F}_i$ be the $\sigma$-field generated by $\{\varepsilon_k:\, k\le i\}$. Given an integrable complex-valued random variable $Z$, we define the following projection operator $\mathcal{P}_i$ as 
\begin{align} \label{proj}
\mathcal{P}_i Z=\E[Z|\mathcal{F}_i]-\E[Z|\mathcal{F}_{i-1}]
\end{align}
for each $i\in\Z$. It is easy to see that 
\begin{align}
\E[\mathcal{P}_i Z\, \mathcal{P}_j W]=0\label{projprod}
\end{align} 
for any two integrable complex-valued random variables $Z$ and $W$ if $i\neq j$.  

The estimation in the following lemma is useful in deriving the order of the bias of the kernel estimator (\ref{kernel}).
\begin{lemma} \label{lma1} Suppose $\sum\limits^{\infty}_{i=0}|a_i|^{\gamma}<\infty$ and $\E|e^{\iota \lambda \varepsilon_1}-\phi_{\varepsilon}(\lambda)|^{2}\leq c_{\gamma,2} \left(|\lambda|^{2\gamma}\wedge 1\right)$ for some $\gamma \in (0,1]$. Then there exists a positive constant $c$ such that
\begin{align*} \label{lminq}
&\sum_{1\leq i\neq j\leq n} \left| \E\left[H(X_i)(\lambda)\overline{H(X_j)(\lambda)}\right]\right| \nonumber\\
&\qquad\qquad \leq c\, n\, |\phi_{\varepsilon}(\lambda a_{0})|\Big[ |\lambda|^{2\gamma}|\phi_{\varepsilon}(\lambda a_{0})|+|\lambda|^{\gamma}|\phi_{\varepsilon}(\lambda a_p)| + 
\sum^p_{r=1}\prod^{\infty}_{\ell=1} |\phi_{\varepsilon}(\lambda (a_{\ell}-a_{\ell+r}))|\Big],
\end{align*}
where $p=\min\{i\in\N:\, a_i\neq 0\}<\infty$. Moreover, the right hand side of the above inequality is equal to $0$ if $a_i=0$ for all $i\in\N$.
\end{lemma}
\begin{proof}  It suffices to show the case when $a_i\neq 0$ for some $i\in\N$. Recall the definition of projection operator $\mathcal{P}_k$ in (\ref{proj}). Applying telescoping, (\ref{projprod}) and the triangle inequality, we have 
\begin{align*}
\sum_{1\leq i\neq j\leq n} \left|\E\left[H(X_i)(\lambda)\overline{H(X_j)(\lambda)}\right]\right|
&\leq 2\sum_{1\leq i<j\leq n}\sum^i_{k=-\infty}\left|\E\left[\mathcal{P}_{k}H(X_i)(\lambda)\mathcal{P}_k\overline{H(X_j)(\lambda)}\right]\right| \\
&=2\sum_{1\leq i<j\leq n}\sum^i_{k=-\infty}\left|\E\left[\mathcal{P}_{0}H(X_{i-k})(\lambda)\mathcal{P}_0\overline{H(X_{j-k})(\lambda)}\right]\right|.
\end{align*}
Note that 
\begin{align*}
\E\left[\mathcal{P}_{0}H(X_{i-k})(\lambda)\mathcal{P}_0\overline{H(X_{j-k})(\lambda)}\right]
&=\E\Big[e^{\iota\sum\limits^{\infty}_{\ell=1}\lambda (a_{\ell+i-k}-a_{\ell+j-k}) \varepsilon_{-\ell}} \Big] \Big[\prod^{i-k-1}_{\ell=0}|\phi_{\varepsilon}(\lambda a_{\ell})|^2\prod^{j-k-1}_{\ell=i-k} \phi_{\varepsilon}(-\lambda a_{\ell}) \Big] \\
&\quad\times\E\left[(e^{\iota  \lambda a_{i-k} \varepsilon_0}-\phi_{\varepsilon}(\lambda a_{i-k}))(e^{-\iota  \lambda a_{j-k} \varepsilon_0}-\phi_{\varepsilon}(-\lambda a_{j-k}))\right].
\end{align*}
Next, we decompose the triple sum $\sum^{n-1}_{i=1}\sum^n_{j=i+1}\sum^{i}_{k=-\infty}$ into the sum with $k<i$, the sum with $k=i$ and $i+2\le j\le n$ and the sum with $k=i$ and $j=i+1$. By Cauchy-Schwartz inequality, we have 
\begin{align*}
&\sum_{1\leq i\neq j\leq n} \left| \E\left[H(X_i)(\lambda)\overline{H(X_j)(\lambda)}\right] \right|\\
&\leq 2\sum^{n-1}_{i=1}\sum^n_{j=i+1}\sum^{i-1}_{k=-\infty} |\phi_{\varepsilon}(\lambda a_{0})|^2 \sqrt{1-|\phi_{\varepsilon}(\lambda a_{i-k})|^2}\sqrt{1-|\phi_{\varepsilon}(\lambda a_{j-k})|^2}\\
&\quad+ 2\sum^{n-1}_{i=1}\sum^n_{j=i+2}\prod^{j-i-1}_{\ell=0} |\phi_{\varepsilon}(\lambda a_{\ell})| \prod^{\infty}_{\ell=1} |\phi_{\varepsilon}(\lambda (a_{\ell}-a_{\ell+j-i}))| \sqrt{1-|\phi_{\varepsilon}(\lambda a_0)|^2} \sqrt{1-|\phi_{\varepsilon}(\lambda a_{j-i})|^2}\\
&\quad\qquad+ 2\sum^{n-1}_{i=1} \Big(|\phi_{\varepsilon}(\lambda a_0)|\prod^{\infty}_{\ell=1} |\phi_{\varepsilon}(\lambda (a_{\ell}-a_{\ell+1}))|\Big).
\end{align*}
Note that the second term on the right hand side of the above inequality is less than
\begin{align*}
&2\sum^{n-1}_{i=1}\sum^n_{j=i+1+p}|\phi_{\varepsilon}(\lambda a_{0})| |\phi_{\varepsilon}(\lambda a_{p})| \sqrt{1-|\phi_{\varepsilon}(\lambda a_{j-i})|^2}\\
&\qquad\qquad+2\sum^{n-1}_{i=1}\sum^{i+p}_{j=i+2}|\phi_{\varepsilon}(\lambda a_{0})| \prod^{\infty}_{\ell=1} |\phi_{\varepsilon}(\lambda (a_{\ell}-a_{\ell+j-i}))| \sqrt{1-|\phi_{\varepsilon}(\lambda a_{j-i})|^2}.
\end{align*}

Therefore, using the conditions of the lemma,
\begin{align*}
\sum_{1\leq i\neq j\leq n} \left| \E\left[H(X_i)(\lambda)\overline{H(X_j)(\lambda)}\right] \right|
&\leq c_1|\lambda|^{2\gamma}\sum^{n-1}_{i=1}\sum^n_{j=i+1}\sum^{i-1}_{k=-\infty} |\phi_{\varepsilon}(\lambda a_0)|^2 |a_{i-k}|^{\gamma}|a_{j-k}|^{\gamma}\\
&\qquad+ c_1|\lambda|^{\gamma}\sum^{n-1}_{i=1}\sum^n_{j=i+1+p}|\phi_{\varepsilon}(\lambda a_{0})||\phi_{\varepsilon}(\lambda a_{p})| |a_{j-i}|^{\gamma}\\
&\qquad\qquad+ c_1\sum^{n-1}_{i=1} |\phi_{\varepsilon}(\lambda a_0)|\sum^p_{r=1}\prod^{\infty}_{\ell=1} |\phi_{\varepsilon}(\lambda (a_{\ell}-a_{\ell+r}))|.
\end{align*}
This gives the desired result.
\end{proof}
\begin{remark}
Note that in Lemma \ref{lma1}, if the linear process $\{X_n\}$ is $m$-dependent, then $a_m\ne 0, m\ge p$ and $\sum\limits_{r=1}^p\prod\limits_{\ell=1}^\infty |\phi_\varepsilon(\lambda(a_\ell-a_{\ell+r}))|\le p|\phi_\varepsilon(\lambda a_m)|$. If $\lim\limits_{i\rightarrow \infty}a_i=0$, $a_0\ne 0$ and $a_i\neq 0$ for infinite many $i\in \N$, then for each $r\in \{1,2,\cdots, p\}$,  $q_r=\min\{i\in \N: a_i-a_{i+r}\ne 0\}<\infty$ and $\sum\limits_{r=1}^p\prod\limits_{\ell=1}^\infty |\phi_\varepsilon(\lambda(a_\ell-a_{\ell+r}))|\le \sum\limits_{r=1}^p |\phi_\varepsilon(\lambda(a_{q_r}-a_{q_r+r}))|$.
\end{remark}

The next lemma is needed in the proof of Theorem \ref{thm2}. 
\begin{lemma}  \label{lma2} 
For any $m\in\N$ and $\gamma \in (0,1]$, if $\E|e^{\iota \lambda \varepsilon_1}-\phi_{\varepsilon}(\lambda)|^{2}\leq c_{\gamma,2} \left(|\lambda|^{2\gamma}\wedge 1\right)$, then there exists a positive constant $c$ such that
\[
\E\Big[\big|\phi_{n}(\lambda)-\phi(\lambda)\big|^2\Big]\leq \frac{c}{n}\Big[ m+2\big(\sum_{i=m}^{\infty}|a_i|^{\gamma}\big)|\lambda|^{\gamma}\prod_{j=0}^{m-1}|\phi_{\varepsilon}(\lambda a_j)|\Big]^2.
\]
\end{lemma}
\begin{proof} 
Obviously, $\E H(X_1)=0$, $\E |H(X_1)|^2 <\infty$ and
\[
n(\phi_{n}(\lambda)-\phi(\lambda))=\sum\limits_{i=1}^{n}\left[e^{\iota \lambda  X_i}-\mathbb{E}\, e^{\iota \lambda  X_i}\right]=\sum\limits_{i=1}^{n} H(X_i).
\]

Note that
\[
\mathcal{P}_{0}H(X_i)=\E[H(X_i)|\mathcal{F}_0]-\E[H(X_i)|\mathcal{F}_{-1}]
\]
for each $i\in\N$. Then,
\begin{align*}
\mathcal{P}_{0}H(X_i)
&=\E[e^{\iota \lambda  X_i}|\mathcal{F}_0]-\E[e^{\iota \lambda  X_i}|
\mathcal{F}_{-1}]\nonumber \\
&=e^{\iota \sum\limits_{j=i}^{\infty} \lambda a_j \varepsilon_{i-j}}\E\Big[e^{\iota \sum\limits_{j=0}^{i-1} \lambda a_j \varepsilon_{i-j}}\Big]-e^{\iota \sum\limits_{j=i+1}^{\infty} \lambda a_j \varepsilon_{i-j}}\E\Big[e^{\iota \sum\limits_{j=0}^{i} \lambda a_j  \varepsilon_{i-j}}\Big]\\
&=e^{\iota \sum\limits_{j=i+1}^{\infty} \lambda a_j  \varepsilon_{i-j}}\E\Big[e^{\iota \sum\limits_{j=0}^{i-1} \lambda a_j \varepsilon_{i-j}}\Big]\Big[e^{\iota  \lambda a_i  \varepsilon_{0}}-\E\, e^{\iota  \lambda   a_i \varepsilon_{0}}\Big]\\
&:=\mbox{I} *\mbox{II}* \mbox{III},
\end{align*}
where $\mbox{I}=e^{\iota\sum\limits_{j=i+1}^{\infty}\lambda  a_j  \varepsilon_{i-j}}$, $\mbox{II}=\E\Big[e^{\iota\sum\limits_{j=0}^{i-1} \lambda a_j  \varepsilon_{i-j}}\Big]=\prod\limits_{j=0}^{i-1}\phi_{\varepsilon}(\lambda a_j)$ and 
\[
\mbox{III}=e^{\iota \lambda  a_i  \varepsilon_{0}}-\E[e^{\iota   \lambda a_i \varepsilon_{0}}]=e^{\iota   \lambda a_i \varepsilon_{0}}-\phi_{\varepsilon}(\lambda a_i).
\]

Since $\mbox{I}$ and $\mbox{III}$ are independent,
\begin{align*}
\|\mathcal{P}_{0}H(X_i)\|^2_2
&=\E\left[\mathcal{P}_{0}H(X_i) \overline{\mathcal{P}_{0}H(X_i)}\right]\\
&=[ \mbox{II}* \overline{\mbox{II}}]\E[ \mbox{I}* \bar{\mbox{I}}] \E[ \mbox{III}* \overline{\mbox{III}}]\\
&=\prod_{j=0}^{i-1}|\phi_{\varepsilon}(\lambda a_j)|^2[1-|\phi_{\varepsilon}(\lambda a_i)|^2]
\end{align*}
for $i\geq 1$ and $\|\mathcal{P}_{0}H(X_0)\|^2_2=1-|\phi_{\varepsilon}(\lambda a_0)|^2$. 
Consequently,
\begin{align*}
\sum_{i=0}^{\infty}\|\mathcal{P}_{0}H(X_i)\|_2
&=\sqrt{1-|\phi_{\varepsilon}(\lambda a_0)|^2}+\sum_{i=1}^{\infty} \prod_{j=0}^{i-1}|\phi_{\varepsilon}(\lambda a_j)| \sqrt{1-|\phi_{\varepsilon}(\lambda a_i)|^2}\\
&\le m+\sum_{i=m}^{\infty} \prod_{j=0}^{i-1}|\phi_{\varepsilon}(\lambda a_j)| \sqrt{1-|\phi_{\varepsilon}(\lambda a_i)|^2}\\
&\le m+\prod_{j=0}^{m-1}|\phi_{\varepsilon}(\lambda a_j)|\sum_{i=m}^{\infty}  \sqrt{1-|\phi_{\varepsilon}(\lambda a_i)|^2}\\
&\le m+c_{\gamma}\Big(\sum_{i=m}^{\infty}|a_i|^{\gamma}\Big) |\lambda|^{\gamma} \prod_{j=0}^{m-1}|\phi_{\varepsilon}(\lambda a_j)|.
\end{align*}
This completes the proof.
\end{proof}

\bigskip

Now we give the proof of Theorem \ref{thm1}.

\medskip

\noindent
\begin{proof}[Proof of Theorem \ref{thm1}] The proof will be done in several steps.

\medskip \noindent
\textbf{Step 1.}
We give the estimation for
\[
\Big|\E[T_n(h_n)]-\int_{\R} f^2(x)\, dx \Big|.
\] 

By Fourier inverse transform, $T_n(h_n)$ can be written as 
\begin{align}\label{FT}
T_n(h_n)
=\frac{1}{\pi n(n-1)}\sum_{1\leq i< j\leq n} \int_{\R}\widehat{K}(\lambda h_n)e^{-\iota \lambda (X_i-X_j)}d\lambda. 
\end{align}
Together with the Plancherel theorem,  we have  
\begin{align}
&\E[T_n(h_n)]-\int_{\R} f^2(x)\, dx\notag\\
&=\frac{1}{2\pi n(n-1)}\sum_{1\leq i\neq j\leq n} \int_{\R}\widehat{K}(\lambda h_n) \, \E\left[ (e^{\iota \lambda X_i}-\phi(\lambda))(e^{-\iota \lambda  X_j}-\phi(-\lambda))\right]\, d\lambda \label{cross}\\
&\qquad+ \frac{1}{2\pi}\int_{\R}\left(\widehat{K}(\lambda h_n)-\widehat{K}(0)\right)|\phi(\lambda)|^2\, d\lambda.\notag
\end{align}

Using Lemma \ref{lma1} and the inequality $|\widehat{K}(\lambda h_n)-\widehat{K}(0)|\leq c_{\beta} |\lambda h_n|^{2\beta}$ for any $\beta\in[0,1]$,
\begin{align*}
\Big| \E[T_n(h_n)]-\int_{\R} f^2(x)\, dx \Big| 
&\leq \frac{c_1}{n} \int_{\R}(1+|\lambda|^{2\gamma})| \phi_{\varepsilon}(\lambda)|^2 \, d\lambda+c_1  h^{2\gamma}_n\int_{\R}|\lambda|^{2\gamma}|\phi(\lambda)|^{2}\, d\lambda\\
&\leq c_2\left(\frac{1}{n}+h^{2\gamma}_n\right).
\end{align*}
In the case that $\{X_i\}_{i=1}^n$ are i.i.d. random variables, the term (\ref{cross}) equals zero and hence
$\Big| \E[T_n(h_n)]-\int_{\R} f^2(x)\, dx \Big|  \leq c_2h^{2\gamma}_n$. This is consistent with the result using convolution method as in Gin\'{e} and Nickl (2008).
\medskip

\medskip \noindent
\textbf{Step 2.} We give the decomposition for 
\[
T_n(h_n)-\E T_n(h_n).
\] 
Using (\ref{FT}),  we can obtain the following decomposition
\begin{align} \label{decomp}
T_n(h_n)-\E T_n(h_n)=2N_n+D_n-\E [D_n],
\end{align}
where 
\begin{align*}
N_n&=\frac{1}{2\pi n}\sum^n_{i=1}\int_{\R} \widehat{K}(\lambda h_n) \big(e^{\iota \lambda X_i}-\phi(\lambda) \big) \phi(-\lambda) \, d\lambda
\end{align*}
and
\begin{align*}
D_n&=\frac{1}{2\pi n(n-1)}\sum_{1\leq i\neq j\leq n}\int_{\R} \widehat{K}(\lambda h_n)\big(e^{\iota \lambda  X_i}-\phi(\lambda) \big)\big(e^{-\iota \lambda  X_j}-\phi(-\lambda) \big) \, d\lambda.
\end{align*}

\medskip \noindent
\textbf{Step 3.} Here we estimate $\E|D_n|^2$.  

We first assume that $\{X_n\}$ is $m$-dependent, that is, $a_{0}\neq 0$, $a_{m}\neq 0$ and $a_i=0$ for all $i>m$. In this case,
\begin{align} \label{d0}
\E |D_n|^2
&\leq \frac{1}{n^2(n-1)^2}\E\Big|\sum_{j-i>4m}\int_{\R} \widehat{K}(\lambda h_n)\big(e^{\iota \lambda  X_i}-\phi(\lambda)\big)\big(e^{-\iota \lambda  X_j}-\phi(-\lambda)\big) \, d\lambda\Big|^2 \nonumber \\  
&\qquad+\frac{1}{n^2(n-1)^2}\E \Big|\sum_{0<j-i\leq 4m}\int_{\R} \widehat{K}(\lambda h_n)\big(e^{\iota \lambda X_i}- \phi(\lambda)\big)\big(e^{-\iota \lambda  X_j}-\phi(-\lambda)\big) \, d\lambda\Big|^2.  
\end{align}

For the first expectation on the right hand side of (\ref{d0}),
\begin{align} \label{d1}
&\E \Big|\sum_{j-i>4m}\int_{\R} \widehat{K}(\lambda h_n)\big(e^{\iota \lambda  X_i}-\phi(\lambda)\big)\big(e^{-\iota \lambda X_j}-\phi(-\lambda)\big) \, d\lambda\Big|^2 \nonumber \\
&\leq \int_{\R^2} |\widehat{K}(\lambda_1 h_n)| |\widehat{K}(\lambda_2 h_n)|  \sum_{|i_1-i_2|\leq m} \Big|\E \big[(e^{\iota \lambda_1  X_{i_1}}-\phi(\lambda_1))(e^{-\iota \lambda_2  X_{i_2}}-\phi(-\lambda_2))\big] \Big| \nonumber \\
&\qquad\times  \sum_{|j_1-j_2|\leq m} \Big|\E \big[(e^{-\iota \lambda_1  X_{j_1}}-\phi(-\lambda_1))(e^{\iota \lambda_2  X_{j_2}}-\phi(\lambda_2))\big]\Big|\, d\lambda_1\, d\lambda_2 \nonumber \\
&\leq\int_{\R^2} |\widehat{K}(\lambda_1 h_n)||\widehat{K}(\lambda_2 h_n)|\, \Big[\sum_{|i_1-i_2|\leq m}\big[|\E e^{\iota \lambda_1  X_{i_1}-\iota \lambda_2 X_{i_2}}|+|\phi(\lambda_1)\phi(-\lambda_2)|\big]\Big]^2\, d\lambda_1\, d\lambda_2 \nonumber \\
&\leq c_3\, n^2\int_{\R^2} |\widehat{K}(\lambda_1h_n)||\widehat{K}(\lambda_2 h_n)|\Big[ |\phi(\lambda_1-\lambda_2)|^2+|\phi_{\varepsilon}(\lambda_1 a_0)\phi_{\varepsilon}(\lambda_2 a_m)|^2+|\phi(\lambda_1)\phi(\lambda_2)|^2\Big]\, d\lambda_1d\lambda_2 \nonumber \\
&\leq c_4\, n^2\int_{\R^2} |\widehat{K}(\lambda_1h_n)|^2 |\phi(\lambda_1-\lambda_2)|^2\, d\lambda_1d\lambda_2+c_4\, n^2\int_{\R^2}  |\widehat{K}(\lambda_2 h_n)|^2 |\phi(\lambda_1-\lambda_2)|^2\, d\lambda_1d\lambda_2 \nonumber \\ 
&\qquad+c_4\, n^2\int_{\R^2}  |\phi_{\varepsilon}(\lambda_1 a_0)\phi_{\varepsilon}(\lambda_2 a_m)|^2\, d\lambda_1d\lambda_2+c_4\, n^2\int_{\R^2}  |\phi(\lambda_1)\phi(\lambda_2)|^2\, d\lambda_1d\lambda_2 \nonumber \\
&\leq c_5\, \frac{n^2}{h_n},
\end{align}
where in the last inequality we used the Plancherel theorem.

Set $K_n(x)=\frac{1}{h_n}K(\frac{x}{h_n})$. Then
\begin{align*}
&\E \Big|\sum_{0<j-i\leq 4m}\int_{\R} \widehat{K}(\lambda h_n)\big(e^{\iota \lambda  X_i}- \phi(\lambda)\big)\big(e^{-\iota \lambda  X_j}-\phi(-\lambda)\big) \, d\lambda \Big|^2 \nonumber \\
&\leq c_6\, n^2\, \sum^{4m}_{k=1} \E \Big|\int_{\R} \widehat{K}(\lambda h_n)\big(e^{\iota \lambda  X_1}- \phi(\lambda)\big)\big(e^{-\iota \lambda  X_{1+k}}-\phi(-\lambda)\big) \, d\lambda\Big|^2 \nonumber \\
&\leq c_7\, n^2\, \sum^{4m}_{k=1}\E K^2_n(X_1-X_{1+k})+ c_7\,n^2\,\E\Big|\int_{\R} \widehat{K}(\lambda h_n)\, e^{\iota \lambda  X_1}\, \phi(-\lambda)\, d\lambda\Big|^2 \nonumber \\
&\qquad+ c_7\, n^2\,\sum^{4m}_{k=1}\E\Big|\int_{\R} \widehat{K}(\lambda h_n)\, e^{-\iota \lambda  X_{1+k}}\phi(\lambda) \, d\lambda \Big|^2+c_7\,n^2\,\Big| \int_{\R}\widehat{K}(\lambda h_n)|\phi(\lambda)|^2 d\lambda \Big|^2.
\end{align*}
Hence
\begin{align}\label{d2}
&\E \Big|\sum_{0<j-i\leq 4m}\int_{\R} \widehat{K}(\lambda h_n)\big(e^{\iota \lambda  X_i}- \phi(\lambda)\big)\big(e^{-\iota \lambda  X_j}-\phi(-\lambda)\big) \, d\lambda \Big|^2 \nonumber \\
&\leq c_8\, n^2\int_{\R} |\widehat{K^2_n}(\lambda)| |\phi_{\varepsilon}(\lambda a_0)| |\phi_{\varepsilon}(\lambda a_m)|\, d\lambda+ c_8\, n^2 \int_{\R} |\widehat{K}(\lambda h_n)|^2\, d\lambda \int_{\R}|\phi(\lambda)|^2 d\lambda \nonumber\\
&\qquad\qquad+c_8\, n^2 \Big(\int_{\R} |\phi(\lambda)|^2\, d\lambda\Big)^2  \nonumber \\
&\leq c_9\, \frac{n^2}{h_n},
\end{align}
where in the last inequality we used the fact that $|\widehat{K^2_n}(\lambda)|\leq \frac{1}{h_n}\int_{\R} K^2(x)\, dx$ for all $\lambda\in\R$.

Combining (\ref{d0}), (\ref{d1}) and (\ref{d2}) gives
\begin{align*}
\E|D_n|^2 \leq \frac{c_{10}}{n^2 h_n}. 
\end{align*}

In the sequel, we can assume that $a_0\ne 0$ and $a_i\neq 0$ for infinite many $i\in \N$. Let $p_1=\min\{i\in\N: a_i\neq 0\}$ and $p_2=\min\{i\in \N: a_i\neq 0, i>p_1\}$. Note that
\begin{align}  \label{L1}
&\|D_n\|_2\leq I_0+I_1+\cdots +I_{p_1}+I_{p_1+1},
\end{align}
where 
\begin{align*}
I_p=\frac{1}{2\pi n(n-1)} \Big\|\int_{\R} \widehat{K}(\lambda h_n)\,  \sum_{1\leq i\neq j\leq n} \sum_{\substack{ k\leq i,\, \ell\leq j\\ |k-i|+|\ell-j|= p}}\big[\mathcal{P}_{k}H(X_i)(\lambda)\mathcal{P}_{\ell}H(X_j)(-\lambda)\big] \, d\lambda\Big\|_2,
\end{align*}
for $0\le p\le p_1$ and 
\begin{align*}
I_{p_1+1}&=\frac{1}{2\pi n(n-1)}  \Big\| \int_{\R} \widehat{K}(\lambda h_n)\, \sum_{1\leq i\neq j\leq n} \sum_{\substack{ k\leq i,\, \ell\leq j\\ |k-i|+|\ell-j|\geq p_1+1}}\big[\mathcal{P}_{k}H(X_i)(\lambda)\mathcal{P}_{\ell}H(X_j)(-\lambda)\big]\, d\lambda \Big\|_2.
\end{align*}

We first estimate $I_0$. Note that 
\begin{align*}
I_{0}=\frac{1}{2\pi n(n-1)} \Big\|\int_{\R} \widehat{K}(\lambda h_n)\,  \sum_{1\leq i\neq j\leq n} \big[\mathcal{P}_{i}H(X_i)(\lambda)\mathcal{P}_{j}H(X_j)(-\lambda)\big] \, d\lambda\Big\|_2,
\end{align*}
and 
\begin{align*}
I^2_0
&\leq \frac{c_{11}}{n^4} \Big\|\int_{\R} \widehat{K}(\lambda h_n)\,  \sum_{j-i=1} \big[\mathcal{P}_{i}H(X_i)(\lambda)\mathcal{P}_{j}H(X_j)(-\lambda)\big] \, d\lambda\Big\|^2_2\\
&\qquad+\frac{c_{11}}{n^4} \Big\|\int_{\R} \widehat{K}(\lambda h_n)\,  \sum_{j-i=2} \big[\mathcal{P}_{i}H(X_i)(\lambda)\mathcal{P}_{j}H(X_j)(-\lambda)\big] \, d\lambda\Big\|^2_2\\
&\qquad+\cdots+ \frac{c_{11}}{n^4} \Big\|\int_{\R} \widehat{K}(\lambda h_n)\,  \sum_{j-i=p_2} \big[\mathcal{P}_{i}H(X_i)(\lambda)\mathcal{P}_{j}H(X_j)(-\lambda)\big] \, d\lambda\Big\|^2_2\\
&\qquad+\frac{c_{11}}{n^4} \Big\|\int_{\R} \widehat{K}(\lambda h_n)\,  \sum_{j-i\geq p_2+1} \big[\mathcal{P}_{i}H(X_i)(\lambda)\mathcal{P}_{j}H(X_j)(-\lambda)\big] \, d\lambda\Big\|^2_2\\
&:=I_{0,1}+I_{0,2}+\cdots+ I_{0, p_2}+ I_{0, p_2+1}.
\end{align*}

Recall properties of the projection operator $\mathcal{P}_k$ defined in (\ref{proj}). It is easy to see that, for $1\le q\le p_2$, 
\begin{align*}
I_{0,q} &\leq  \frac{c_{12}}{n^3} \int_{\R^2} |\widehat{K}(\lambda_1 h_n)||\widehat{K}(\lambda_2 h_n)|\, \Big|\prod^{\infty}_{k=1} \E [e^{\iota(\lambda_2-\lambda_1)  (a_{k+q}-a_k) \varepsilon_{j-k-q}}] \Big|\, d\lambda_1\, d\lambda_2\\
&= \frac{c_{12}}{n^3} \int_{\R^2} |\widehat{K}(\lambda_1 h_n)||\widehat{K}(\lambda_2 h_n)|\prod^{\infty}_{k=1}|\phi_{\varepsilon}((\lambda_2-\lambda_1)(a_{k+q}-a_k))|\, d\lambda_1\, d\lambda_2.
\end{align*}
Since $\lim\limits_{i\to\infty} a_i=0$ and $a_i\neq 0$ for infinite many $i\in\N$, there exist infinite many indices $k$ such that $a_{k+q}\neq a_k$. Doing some simple calculations then gives
\begin{align*}
I_{0,q}
&\leq  \frac{c_{13}}{n^3h_n}, 
\end{align*}
for $1\le q\le p_2$. To estimate the last term in the upper bound of $I_0^2$, observe that
\begin{align*}
I_{0,p_2+1}
&=\frac{c_{11}}{n^4}\int_{\R^2} \widehat{K}(\lambda_1 h_n)\widehat{K}(-\lambda_2 h_n)\\
&\qquad \sum_{\substack{ j-i_1\geq p_2+1,\; j-i_2\geq p_2+1 }} \E \left[\mathcal{P}_{i_1} H(X_{i_1})(\lambda_1)\mathcal{P}_{i_2} H(X_{i_2})(-\lambda_2)\mathcal{P}_j H(X_j)(-\lambda_1)\mathcal{P}_j H(X_j)(\lambda_2)\right]\, d\lambda_1\, d\lambda_2\\
&=\frac{c_{11}}{n^4}\int_{\R^2} \widehat{K}(\lambda_1 h_n)\widehat{K}(-\lambda_2 h_n) I_1(\lambda_1,\lambda_2)\, d\lambda_1\, d\lambda_2+ \frac{c_{11}}{n^4}\int_{\R^2} \widehat{K}(\lambda_1 h_n)\widehat{K}(\lambda_2 h_n) I_2(\lambda_1,\lambda_2)\, d\lambda_1\, d\lambda_2\\
&\qquad\qquad\qquad\qquad+ \frac{c_{11}}{n^4}\int_{\R^2} \widehat{K}(\lambda_1 h_n)\widehat{K}(-\lambda_2 h_n) I_3(\lambda_1,\lambda_2)\, d\lambda_1\, d\lambda_2,
\end{align*}
where 
\begin{align*}
I_1(\lambda_1,\lambda_2)&=\sum_{\substack{j-i_1\geq p_2+1,\; j-i_2\geq p_2+1,\; i_2>i_1}} \E \left[\mathcal{P}_{i_1} H(X_{i_1})(\lambda_1)\mathcal{P}_{i_2} H(X_{i_2})(-\lambda_2)\mathcal{P}_j H(X_j)(-\lambda_1)\mathcal{P}_j H(X_j)(\lambda_2)\right],  \\
I_2(\lambda_1,\lambda_2)&=\sum_{\substack{j-i_1\geq p_2+1,\; j-i_2\geq p_2+1,\; i_1<i_2}} \E \left[\mathcal{P}_{i_1} H(X_{i_1})(\lambda_1)\mathcal{P}_{i_2} H(X_{i_2})(-\lambda_2)\mathcal{P}_j H(X_j)(-\lambda_1)\mathcal{P}_j H(X_j)(\lambda_2)\right],   \\
I_3(\lambda_1,\lambda_2)&=\sum_{\substack{ j-i_1\geq p_2+1,\; j-i_2\geq p_2+1,\; i_2=i_1}} \E \left[\mathcal{P}_{i_1} H(X_{i_1})(\lambda_1)\mathcal{P}_{i_2} H(X_{i_2})(-\lambda_2)\mathcal{P}_j H(X_j)(-\lambda_1)\mathcal{P}_j H(X_j)(\lambda_2)\right].
\end{align*}

Note that 
\begin{align*}
\left|I_1(\lambda_1,\lambda_2)\right|
&\leq \sum_{\substack{ j-i_1\geq p_2+1,\; j-i_2\geq p_2+1,\; i_2>i_1}}  \left|\phi_{\varepsilon}(a_{p_1}(\lambda_1-\lambda_2))\right|\left|\phi_{\varepsilon}(a_{p_2}(\lambda_1-\lambda_2))\right|\\
&\qquad\qquad\times \left| \E \left[ (e^{-\iota \lambda_2 a_0 \varepsilon_{i_2}}-\phi_{\varepsilon}(- \lambda_2 a_0)) (e^{-\iota  (\lambda_1-\lambda_2) a_{j-i_2} \varepsilon_{i_2}} -\phi_{\varepsilon}(-(\lambda_1-\lambda_2) a_{j-i_2}) \right] \right|\\
&\leq  \sum_{\substack{ j-i_1\geq p_2+1,\; j-i_2\geq p_2+1,\; i_2>i_1}} \left|\phi_{\varepsilon}((\lambda_1-\lambda_2) a_{p_1})\right|\left|\phi_{\varepsilon}((\lambda_1-\lambda_2) a_{p_2})\right|\\
&\qquad\qquad\times\left\| e^{-\iota  (\lambda_1-\lambda_2) a_{j-i_2} \varepsilon_{i_2}} -\phi_{\varepsilon}(-(\lambda_1-\lambda_2) a_{j-i_2})\right\|_2 \\
&\leq c_{16}\, \sum_{\substack{j-i_1\geq p_2+1,\; j-i_2\geq p_2+1,\; i_2>i_1}}  \left|\phi_{\varepsilon}((\lambda_1-\lambda_2)a_{p_1} )\right|\left|\phi_{\varepsilon}((\lambda_1-\lambda_2) a_{p_2})\right|\, |\lambda_1-\lambda_2|^{\gamma} |a_{j-i_2}|^{\gamma}\\
&\leq c_{17}\, n^2\, |\lambda_1-\lambda_2|^{\gamma}  \left|\phi_{\varepsilon}((\lambda_1-\lambda_2)a_{p_1})\right|\left|\phi_{\varepsilon}((\lambda_1-\lambda_2)a_{p_2})\right|.
\end{align*}

Similarly,
\begin{align*}
\left|I_2(\lambda_1,\lambda_2)\right|
&\leq c_{18}\, n^2\, |\lambda_1-\lambda_2|^{\gamma}  |\phi_{\varepsilon}((\lambda_1-\lambda_2)a_{p_1})| |\phi_{\varepsilon} ((\lambda_1-\lambda_2)a_{p_2}) |.
\end{align*}
Finally,
\begin{align*}
\left| I_3(\lambda_1,\lambda_2)\right|
&\leq \sum_{j-i\geq p_2+1} \Big| \E \left[\mathcal{P}_{i} H(X_{i})(\lambda_1)\mathcal{P}_{i} H(X_{i})(-\lambda_2)\mathcal{P}_j H(X_j)(-\lambda_1)\mathcal{P}_j H(X_j)(\lambda_2)\right] \Big|\\
&\leq c_{19}\, n^2 |\phi_{\varepsilon}((\lambda_1-\lambda_2)a_{p_1})||\phi_{\varepsilon} ((\lambda_1-\lambda_2)a_{p_2})|.
\end{align*}
So
\begin{align*}
I_{0,p_2+1}
&\leq \frac{c_{20}}{n^2} \int_{\R^2} |\widehat{K}(\lambda_1 h_n)| |\widehat{K}(\lambda_2 h_n)|(1+|\lambda_1-\lambda_2|^{\gamma})  |\phi_{\varepsilon}((\lambda_1-\lambda_2) a_{p_1})| |\phi_{\varepsilon} ((\lambda_1-\lambda_2) a_{p_2})| \, d\lambda_1\, d\lambda_2\\
&\leq \frac{c_{21}}{n^2h_n}.
\end{align*}
Therefore, 
\begin{align} \label{I1}
I^2_0\leq I_{0,1}+I_{0,2}+\cdots +I_{0, p_2}+ I_{0,p_2+1}\leq \frac{c_{22}}{n^2 h_n}.
\end{align}

Using similar arguments, for $1\le p\le p_1$, we can obtain 
\begin{align} \label{I2}
I^2_p\leq \frac{c_{23}}{ n^2 h_n}.
\end{align}

We now estimate $I_{p_1+1}$.  Note that 
\begin{align*}
I^2_{p_1+1}=\frac{1}{(2\pi)^2 n^2(n-1)^2} \int_{\R^2} \widehat{K}(\lambda_1 h_n)\widehat{K}(-\lambda_2 h_n)\, J_n(\lambda_1,\lambda_2)\, d\lambda_1\, d\lambda_2,
\end{align*}
where 
\begin{align*}
J_n(\lambda_1,\lambda_2)
&=\sum_{\substack{ i_1\neq j_1,\; i_2\neq j_2\\ k_1\leq i_1,\; k_2\leq j_1,\; k_3\leq i_2,\; k_4\leq j_2\\ |k_1-i_1|+|k_2-j_1|\geq p_1+1\\ |k_3-i_2|+|k_4-j_2|\geq p_1+1}}\E \Big[\mathcal{P}_{k_1} H(X_{i_1})(\lambda_1) \mathcal{P}_{k_2} H(X_{j_1})(-\lambda_1) \mathcal{P}_{k_3} H(X_{i_2})(-\lambda_2)\mathcal{P}_{k_4} H(X_{j_2})(\lambda_2)\Big].
\end{align*}
Recall properties of the projection operator $\mathcal{P}_k$ defined in (\ref{proj}). By Cauchy-Schartz inequality and the condition (\ref{char4moment}), we can obtain
\begin{align}\label{cova}
|J_n(\lambda_1,\lambda_2)|
&\leq c_{24}\; |\lambda_1|^{2\gamma}|\lambda_2|^{2\gamma} \Big(|\phi_{\varepsilon}(\lambda_1 a_0)|^2+|\phi_{\varepsilon}(\lambda_1 a_{p_1})|^2\Big) \, \Big(|\phi_{\varepsilon}(\lambda_2 a_0)|^2+|\phi_{\varepsilon}(\lambda_2 a_{p_1})|^2\Big) \nonumber \\  
&\qquad \times \sum_{\substack{i_1\neq j_1,\; i_2\neq j_2\\ k_1\leq i_1,\; k_2\leq  j_1,\; k_3\leq i_2\; k_4\leq j_2\\
|k_1-i_1|+|k_2-j_1|\geq p_1+1,\, |k_3-i_2|+|k_4-j_2|\geq p_1+1\\
\text{at least two}\, k\, \text{indices are identical,}\\
\text{identical}\, k\, \text{indices are greater than the others}}} |a_{i_1-k_1}|^{\gamma}|a_{j_1-k_2}|^{\gamma}|a_{i_2-k_3}|^{\gamma}|a_{j_2-k_4}|^{\gamma}.
\end{align}

There are only four possibilities for the $k$ indices in the above summation:  1) all the $k$ indices are identical; 2) three $k$ indices are identical and strictly larger than the last one; 3) two pairs of different identical $k$ indices; 4) one pair of identical $k$ indices and strictly larger than the remaining two different indices. For the first three possibilities, the summation in (\ref{cova}) is less than a constant multiple of $n^2$ since $\sum\limits^{\infty}_{i=0}|a_i|^{\gamma}<\infty$. For the last possibility, the summation in (\ref{cova}) is less than a constant multiple of $n^2\eta_{n,\gamma}$ where $\eta_{n,\gamma}=\sum\limits^{n}_{\ell=0}\sum\limits^{\infty}_{i=\ell} |a_i|^{\gamma}$. Therefore,
\begin{align}  \label{I3}
I^2_{p_1+1}
&\leq \frac{c_{25}\, \eta_{n,\gamma}}{n^2} \left( \int_{\R} |\widehat{K}(\lambda h_n)|\, |\lambda|^{2\gamma}\, \Big(|\phi_{\varepsilon}(\lambda a_0)|^2+|\phi_{\varepsilon}(\lambda a_{p_1})|^2\Big)\, d\lambda\right)^2.
\end{align}

Putting inequalities (\ref{L1}), (\ref{I1}), (\ref{I2}) and (\ref{I3}) together gives 
\begin{align*}
\|D_n\|^2_2\leq c_{26} \left(\frac{1}{n^2 h_n}+\frac{\eta_{n,\gamma}}{n^2} \right).
\end{align*}

\medskip \noindent
\textbf{Step 4.} Here we estimate 
\[
\E\big[|N_n-\overline{N}_n|^2\big],
\]
where 
\[
\overline{N}_n=\frac{1}{2\pi}\int_{\R} \widehat{K}(0) \big(\phi_n(\lambda)-\phi(\lambda) \big) \phi(-\lambda) \,  d\lambda.
\]

Note that 
\begin{align*} 
\E\left[|N_n-\overline{N}_n|^2\right]
&=\E\left[ \Big|\frac{1}{2\pi}\int_{\R} \big(\widehat{K}(\lambda h_n)-\widehat{K}(0)\big)\big(\phi_n(\lambda)-\phi(\lambda) \big) \phi(-\lambda) \,  d\lambda\Big|^2\right]\\
&\leq \frac{1}{(2\pi)^2n} \, \E\left[ \Big|\int_{\R} \big[\widehat{K}(\lambda h_n)-\widehat{K}(0)\big]\, e^{\iota\lambda  X_1}\phi(-\lambda) \,  d\lambda\Big|^2\right]\\
&= \frac{1}{n} \, \E\Big[K_n* f(X_1)-f(X_1)\Big]^2,
\end{align*}
where $K_n(x)=\frac{1}{h_n}K(\frac{x}{h_n})$.

Since $f$ is bounded,
\begin{align*} 
\E\left[|N_n-\overline{N}_n|^2\right]
&\leq \frac{ c_{27}}{n} \int_{\R} \big[ K_n* f(x)-f(x) \big]^2\, dx \\
&= \frac{c_{27}}{(2\pi)^2 n} \int_{\R}|\widehat{K}(\lambda h_n)-\widehat{K}(0)|^2|\phi(\lambda)|^2\, d\lambda \\
&\leq \frac{c_{27}}{(2\pi)^2 n} \int_{\R}|\widehat{K}(\lambda h_n)-\widehat{K}(0)|^2|\phi_{\varepsilon}(a_0\lambda)|^2\, d\lambda \\
&\leq \frac{c_{28}\, h^{2\gamma}_n}{n},
\end{align*}
where in the first equality we used the Plancherel theorem.

It is easy to see that 
\[
\overline{N}_n
=\frac{1}{n} \sum\limits^n_{i=1}\left(f(X_i)-\int_{\R} f^2(x)\, dx\right).
\] 
Combining all the above results gives 
\[
\E\Big|T_n(h_n)-\E T_n(h_n)-\frac{1}{n}\sum^n_{i=1}Y_i\Big|^2\leq c_{29}\left(\frac{1}{n^2 h_n}+\frac{\eta_{n,\gamma}}{n^2}+\frac{h^{2\gamma}_n}{n}\right).
\]

\medskip \noindent
\textbf{Step 5.}  We show the central limit theorem for
\[
\frac{1}{\sqrt{n}} \sum^n_{i=1}\Big(f(X_i)-\int_{\R} f^2(x)\, dx\Big).
\]
It suffices to consider the case when the $X_i$ are dependent. Note that 
\begin{align*}
&\mathcal{P}_{0}\big[f(X_i)-\int_{\R} f^2(x)\, dx\big]\\
&=\E [f(X_i)|\mathcal{F}_0]-\E [f(X_i)|\mathcal{F}_{-1}] \\
&=\frac{1}{2\pi}\, \E\bigg[\int_{\R} \phi(\lambda) \, e^{\iota\sum\limits_{j=i+1}^{\infty} \lambda a_j \varepsilon_{i-j}+\iota\sum\limits_{j=0}^{i-1}  \lambda a_j \varepsilon_{i-j}}\left[e^{\iota  \lambda a_i \varepsilon_0}-e^{\iota \lambda a_i  \varepsilon'_0}\right]\, d\lambda \, \bigg|\mathcal{F}_0\bigg], 
\end{align*}
where $\{\varepsilon'_i: i\in\Z\}$ is an independent copy of $\{\varepsilon_i: i\in\Z\}$.

Then
\begin{align*}
\Big\|\mathcal{P}_{0}\big[f(X_i)-\int_{\R} f^2(x)\, dx\big]\Big\|_2
&\leq \Big\|\int_{\R} \phi(\lambda) \, e^{\iota \sum\limits_{j=i+1}^{\infty} \lambda a_j \varepsilon_{i-j}+\iota \sum\limits_{j=0}^{i-1} \lambda a_j \varepsilon_{i-j}}\big[e^{\iota  \lambda a_i \varepsilon_0}-e^{\iota \lambda a_i \varepsilon'_0}\big]\, d\lambda\Big\|_2\\
&\leq \Big\|\int_{\R} |\phi(\lambda)| |e^{\iota  \lambda a_i \varepsilon_0}-e^{\iota \lambda a_i \varepsilon'_0}|\, d\lambda\Big\|_2\\
&\leq 2\Big\|\int_{\R} |\phi(\lambda)| |e^{\iota  \lambda a_i \varepsilon_0}-\phi_\varepsilon(\lambda a_i)|\, d\lambda\Big\|_2\\
&\leq 2|a_i|^{\gamma} \int_{\R} |\lambda|^{\gamma} |\phi(\lambda)|\, d\lambda\\
&\leq c_{30} \, |a_i|^{\gamma}\int_{\R} |\lambda|^{\gamma} |\phi_{\varepsilon}(\lambda)|^2\, d\lambda.
\end{align*}

Now by Lemma 1 in Wu (2006), we can easily obtain
\[
\frac{1}{\sqrt{n}} \sum^n_{i=1}\Big(f(X_i)-\int_{\R} f^2(x)\, dx\Big)\overset{d}{\longrightarrow} N(0,\sigma^2)
\]
for some $\sigma^2\in(0,\infty)$. Using similar arguments as in \textbf{Step 3}, we can show that 
\[
\sup_{n}\frac{1}{n^2}\E|S_n|^4<\infty,
\] 
where $S_n=\sum\limits^n_{i=1}\left(f(X_i)-\int_{\R} f^2(x)\, dx\right)$.

Hence, $\sigma^2=\lim\limits_{n\to\infty} n^{-1}\Var(S_n)$. This completes the proof.
\end{proof}

\bigskip

\noindent
\begin{proof}[Proof of Theorem \ref{thm2}] The proof of this theorem is similar to that of Theorem \ref{thm1}. We only need to modify the \textbf{Step 3} in the proof of Theorem \ref{thm1} and consider the dependent case. Using the same notations as in the proof of Theorem \ref{thm1} and then Lemma \ref{lma2}, 
\begin{align*} 
\sqrt{n}\, \E |D_n|
&\leq \sqrt{n}\, \int_{\R} |\widehat{K}(\lambda h_n)| \E\left[|\phi_n(\lambda)-\phi(\lambda)|^2\right] \, d\lambda \\
&\qquad\qquad +\frac{1}{\sqrt{n}}\int_{\R} |\widehat{K}(\lambda h_n)|\E\left[|e^{\iota \lambda  X_1}-\E e^{\iota \lambda  X_1}|^2\right] \, d\lambda \\
&\leq \frac{c_1}{\sqrt{n}h_n}\int_{\R} |\widehat{K}(\lambda )|\, d\lambda+\frac{c_1}{\sqrt{n}}\int_{\R} |\lambda|^{2\gamma} |\phi_{\varepsilon}(\lambda a_0)|^2\, d\lambda.
\end{align*}
This completes the proof.
\end{proof}

\bigskip

At the end, we give two useful lemmas which are related to the sufficient and necessary conditions for the assumptions in Theorems \ref{thm1} and \ref{thm2}.

\begin{lemma} \label{lma} Let $\phi_{\varepsilon}(\lambda)=\E[e^{\iota \lambda \varepsilon_1}]$ for all $\lambda\in\R$.  For any $\gamma\in[0,1]$, if $\E|\varepsilon_1|^{2\gamma}<\infty$, then there exists a positive constant $c_{\gamma,2}$ such that 
\[
\E |e^{\iota \lambda \varepsilon_1}-\phi_{\varepsilon}(\lambda)|^2\leq c_{\gamma,2}\, |\lambda|^{2\gamma}.
\]
Moreover, if $\E|\varepsilon_1|^{4\gamma}<\infty$, there exists a positive constant $c_{\gamma,4}$ such that
\[
\E |e^{\iota \lambda  \varepsilon_1}-\phi_{\varepsilon}(\lambda)|^4\leq c_{\gamma,4}\, |\lambda|^{4\gamma}.
\]
\end{lemma}
\begin{proof} This first inequality follows from 
\[
\E  |e^{\iota \lambda \varepsilon_1}-\phi_{\varepsilon}(\lambda)|^2
=1-|\phi_{\varepsilon}(\lambda)|^2
\leq 1-[\E\cos(\lambda \varepsilon_1)]^2\leq c_1|\lambda|^{2\gamma}.
\]
For the second inequality,
\begin{align*}
&\E |e^{\iota \lambda \varepsilon_1}-\phi_{\varepsilon}(\lambda)|^4\\
&=\E \big[ 1-e^{\iota \lambda  \varepsilon_1} \phi_{\varepsilon}(-\lambda)-e^{-\iota \lambda  \varepsilon_1} \phi_{\varepsilon}(\lambda)+|\phi_{\varepsilon}(\lambda)|^2\big]^2\\
&=1+\phi_{\varepsilon}(2\lambda)\phi^2_{\varepsilon}(-\lambda)+\phi_{\varepsilon}(-2\lambda)\phi^2_{\varepsilon}(\lambda)-3|\phi_{\varepsilon}(\lambda)|^4\\
&=1+ 2 \E\cos(2\lambda \varepsilon_1) \big[\E^2\cos(\lambda \varepsilon_1)-\E^2\sin(\lambda  \varepsilon_1) \big]+4\E\sin(2\lambda  \varepsilon_1)\E\cos(\lambda  \varepsilon_1)\E\sin(\lambda  \varepsilon_1)\\
&\qquad -3\big[\E^2\cos(\lambda \varepsilon_1)+\E^2\sin(\lambda \varepsilon_1)\big]^2\\
&=1+ 2 \big[1-2\E\sin^2(\lambda \varepsilon_1)\big] \big[(1-2\E\sin^2(\frac{\lambda \varepsilon_1}{2}))^2-\E^2\sin(\lambda \varepsilon_1) \big]\\
&\qquad +8\E\big[\sin(\lambda \varepsilon_1)(1-2\sin^2(\frac{\lambda \varepsilon_1}{2}))\big]\E\sin(\lambda \varepsilon_1)[1-2\E\sin^2(\frac{\lambda \varepsilon_1}{2})]\\
&\qquad -3\big[(1-2\E\sin^2(\frac{\lambda \varepsilon_1}{2}))^2+\E^2\sin(\lambda  \varepsilon_1)\big]^2\\
&\leq c_2 |\lambda|^{4\gamma}.
\end{align*}
This completes the proof.
\end{proof}

\begin{lemma} \label{optimal}
Let $\phi_{\varepsilon}(\lambda)=\E[e^{\iota \lambda \varepsilon_1}]$ for all $\lambda\in\R$. If $\varepsilon_1$ has non-degenerate distribution, i.e., it is not equal to a constant almost surely, then in the inequalities
$$\E|e^{\iota\lambda\varepsilon_1}-\phi_{\varepsilon}(\lambda)|^2\leq c_{\gamma,2}(|\lambda|^{2\gamma}\wedge 1)$$ 
and 
$$\E|e^{\iota\lambda\varepsilon_1}-\phi_{\varepsilon}(\lambda)|^4\leq c_{\gamma,4}(|\lambda|^{4\gamma}\wedge 1),$$  
the range $\gamma\in(0,1]$ is optimal. 
\end{lemma}
\begin{proof}
By Cauchy-Schwarz inequality, we only need to show that the range $\gamma\in(0,1]$ is optimal for the first inequality. Suppose $\E|e^{\iota \lambda \varepsilon_1}-\phi_{\varepsilon}(\lambda)|^2=1-|\phi_{\varepsilon}(\lambda)|^2\leq c(|\lambda|^{2\gamma}\wedge 1)$ for some $\gamma>1$. Since $\varepsilon_2$ is an independent copy of $\varepsilon_1$,  let $\phi_{\varepsilon_1-\varepsilon_2}(\lambda)=\E[e^{\iota \lambda (\varepsilon_1-\varepsilon_2)}]$,  then $1-\phi_{\varepsilon_1-\varepsilon_2}(\lambda)=1-|\phi_{\varepsilon}(\lambda)|^2\leq c(|\lambda|^{2\gamma}\wedge 1)$ for some $\gamma>1$. 

Note that
\begin{align*}
\E|e^{\iota \lambda (\varepsilon_1-\varepsilon_2)} -1|^2
&=1-\phi_{\varepsilon_1-\varepsilon_2}(\lambda)+1-\phi_{\varepsilon_1-\varepsilon_2}(-\lambda)=2(1-\phi_{\varepsilon_1-\varepsilon_2}(\lambda)).
\end{align*}
Therefore, for $\lambda$ close enough to 0,
\begin{align*}
|\phi_{\varepsilon_1-\varepsilon_2}(x+\lambda)-\phi_{\varepsilon_1-\varepsilon_2}(x)|
&\leq \E|e^{\iota \lambda \varepsilon_1-\varepsilon_2} -1|\leq \left( \E|e^{\iota \lambda \varepsilon_1-\varepsilon_2} -1|^2\right)^{\frac{1}{2}}\leq c_1 |\lambda|^{\gamma}.
\end{align*}
Since $\gamma>1$, $\phi_{\varepsilon_1-\varepsilon_2}'(x)=0$ for all $x\in\R$. That is, $|\phi_{\varepsilon}(x)|^2=\phi_{\varepsilon_1-\varepsilon_2}(x)=1$ for all $x\in\R$ or $\varepsilon_1$ equals a constant almost surely. 
So the $\gamma$ must be less than or equal to 1. 
\end{proof}

\section*{Acknowledgement}

We would like to thank the two referees and the Associate Editor for their valuable comments. We would also like to thank Professor Xia Chen for very helpful discussions. F. Xu is partially supported by National Natural  Science Foundation of China (Grant No.11401215), Shanghai Pujiang Program (14PJ1403300), 111 Project (B14019) and Natural Science Foundation of Shanghai (16ZR1409700).

\end{document}